\documentclass{article}
\usepackage[english]{babel}
\usepackage{booktabs}
\usepackage[pdftex]{graphicx}
\usepackage{epstopdf}
\usepackage{mathtools}
\usepackage{sidecap}
\usepackage{hyperref}
\hypersetup{
	colorlinks,
	linkcolor={red!50!black},
	citecolor={blue!50!black},
	urlcolor={blue!80!black}
}
\usepackage{amssymb, amsmath,amsthm}
\usepackage{dsfont}
\usepackage{diagbox,multirow}
\usepackage{hhline}
\usepackage[dvipsnames]{xcolor}
\usepackage{array}
\usepackage{colortbl}
\usepackage{subfig}
\usepackage{enumerate}
\usepackage[shortlabels]{enumitem}
\usepackage{makecell}
\usepackage{float}
\usepackage{framed}
\usepackage{varwidth}
\usepackage[export]{adjustbox}
\usepackage{xparse}
\usepackage {tikz}
\usetikzlibrary{arrows,decorations.pathmorphing,backgrounds,positioning,fit,matrix,patterns}
\usetikzlibrary{decorations.markings}
\usepackage{todonotes}

\usepackage{amsbsy}

\pgfdeclarelayer{myback}
\pgfsetlayers{myback,background,main}
\definecolor {processblue}{cmyk}{0.96,0,0,0}
\definecolor {darkred}{RGB}{190,11,0}

\tikzset{mycolor/.style = {line width=1bp,color=#1}}%
\tikzset{myfillcolor/.style = {draw,fill=#1}}%

\newlength\mylen
\tikzset{
	bicolor/.style 2 args={
		dashed,dash pattern=on 20pt off 20pt,-,#1,
		postaction={draw,dashed,dash pattern=on 20pt off 20pt,-,#2,dash phase=20pt}
	},
}

\NewDocumentCommand{\fhighlight}{O{blue!40} m m}{%
	\draw[myfillcolor=#1] (#2.north west)rectangle (#3.south east);
}

\usetikzlibrary{fit}
\tikzset{%
	highlight/.style={rectangle,rounded corners,fill=red!15,draw,
		fill opacity=0.5,thick,inner sep=0pt}
}

\DeclarePairedDelimiter\floor{\lfloor}{\rfloor}

\pgfdeclarepatternformonly{mynewdots}{\pgfqpoint{-2pt}{-2pt}}{\pgfqpoint{2pt}{2pt}}{\pgfqpoint{2pt}{2pt}}
{%
  \pgfpathcircle{\pgfqpoint{0pt}{0pt}}{.5pt}%
  \pgfusepath{fill}%
}%

\newtheorem{thm}{Theorem}[section]

\newtheorem{prop}[thm]{Proposition}

\newtheorem{defn}{Definition}[section]

\newtheorem{rem}{Remark}[section]

\newcommand{\N}{\mathbb{N}}
\newcommand{\R}{\mathbb{R}}

\newcommand{\diag}{\mathop{\mathrm{diag}}}\renewcommand{\i}{\textup{i}}\newcommand{\bfzero}{\mathbf 0}\newcommand{\bfuno}{\mathbf 1}

\newcommand{\bfnn}{{\boldsymbol n}}\newcommand{\bfff}{{\boldsymbol f}}\newcommand{\bftt}{{\boldsymbol t}}\newcommand{\bfii}{{\boldsymbol i}}\newcommand{\bfjj}{{\boldsymbol j}}\newcommand{\bfkk}{{\boldsymbol k}}\newcommand{\bfhh}{{\boldsymbol h}}\newcommand{\bfww}{{\boldsymbol w}}\newcommand{\bfxx}{{\boldsymbol x}}
\newcommand{\bftheta}{{\boldsymbol\theta}}
\newcommand{\bfeta}{{\boldsymbol\eta}}

\newcommand*\quot[2]{{^{\textstyle #1}\big/_{\textstyle #2}}}
\DeclareMathSymbol{\shortminus}{\mathbin}{AMSa}{"39}

\newcommand{\blue}{\textcolor{blue}}

\makeatletter
\newcommand{\subjclass}[2][1991]{%
	\let\@oldtitle\@title%
	\gdef\@title{\@oldtitle\footnotetext{#1 \emph{Mathematics subject classification.} #2}}%
}
\newcommand{\keywords}[1]{%
	\let\@@oldtitle\@title%
	\gdef\@title{\@@oldtitle\footnotetext{\emph{Key words and phrases.} #1.}}%
}
\makeatother
\begin{document}
\title{Asymptotic spectra of large (grid) graphs with a uniform local structure (part II): numerical applications}
\author{
Andrea Adriani $^{\dagger, (a)}$, Davide Bianchi $^{\dagger, (b)}$, Paola Ferrari $^{\dagger, (c)}$\\ Stefano Serra-Capizzano $^{\dagger, (c,d)}$}
\date{}

\maketitle

\begin{abstract}
In the current work we are concerned with sequences of graphs having a grid geometry,
with a uniform local structure in a bounded domain $\Omega\subset {\mathbb
	R}^d$, $d\ge 1$. When
$\Omega=[0,1]$, such graphs include the standard Toeplitz graphs
and, for $\Omega=[0,1]^d$,  the considered class includes
$d$-level Toeplitz graphs. In the general case, the underlying
sequence of adjacency matrices has a canonical eigenvalue
distribution, in the Weyl sense, and it has been shown in the theoretical part of this work that we can associate
to it a symbol $\boldsymbol{\mathfrak{f}}$. The knowledge of the symbol and of its basic analytical features provides key information on the eigenvalue
structure in terms of  localization, spectral gap, clustering, and global distribution.
In the present paper, many different applications are discussed and various numerical examples are presented in order to underline the practical use of the developed theory. Tests and applications are mainly obtained from the
approximation of differential operators via numerical schemes such as Finite Differences (FDs), Finite Elements (FEs), and Isogeometric Analysis (IgA). Moreover, we show that more applications can be taken into account, since the results presented here can be applied as well to study the spectral properties of adjacency matrices and Laplacian operators of general large graphs and networks, whenever the involved matrices enjoy a uniform local structure.
\end{abstract}
\noindent
$\dagger$ All authors have contributed equally.\\
(a)\, Department of Theoretical and Applied Sciences, University of Insubria, Via Dunant 3, 21100 Varese, Italy.\\
(aadriani@uninsubria.it)\\
(b) \, School of Science,   Harbin Institute of Technology (Shenzhen), University Town of Shenzhen, 518055 Shenzhen, China.\\
(bianchi@hit.edu.cn)\\
(c) \, Department of Science and High Technology,   University of Insubria,  Via Valleggio 11, 22100 Como, Italy.\\
(pferrari@uninsubria.it, stefano.serrac@uninsubria.it)\\
(d)\, Department of Information Technology, Uppsala University, Uppsala, Sweeden \\ (stefano.serra@it.uu.se)\\

\section{Introduction}


In \cite{twin-theory} a large class of (sequences of) graphs has been defined, according to the following structural properties:
\begin{description}
\item[a)] when we
look at them from ``far away'', they should reconstruct
approximately a given domain $\Omega\subset [0,1]^d$, $d\ge 1$,
i.e., the larger is the number of the nodes the more accurate is
the reconstruction of $\Omega$;
\item[b)] when we look at them
``locally'', that is from a generic internal node, we want that
the structure is uniform, i.e., we should be unable to understand
where we are in the graphs, except possibly when the considered
node is close enough to the boundaries of $\Omega$.
\end{description}

In particular, the domain $\Omega$ and the internal structure of the graphs in the sequence are
fixed, independently of the index (or multi-index) of the graph
uniquely related to the cardinality of nodes: it follows that the resulting
sequence of graphs has a grid geometry, with a uniform local
structure, in a bounded domain $\Omega\subset {\mathbb R}^d$, $d\ge 1$.  The domain $\Omega$ is assumed to be Lebesgue measurable with regular boundary, that is a boundary $\partial \Omega$ of zero Lebesgue measure,
and it is supposed to be contained in the hypercube $[0,1]^d$. Such a domain is called {\itshape regular}. When $\Omega=[0,1]$, it has been proven in \cite{twin-theory} that such graphs include the standard
Toeplitz graphs (see \cite{Ghorban2012}) and for $\Omega=[0,1]^d$  the
considered class includes $d$-level Toeplitz graphs.

Moreover, given a sequence of graphs having  a grid geometry with a uniform local structure in a domain $\Omega$, the underlying sequence
of adjacency matrices shows a canonical eigenvalue distribution, in the Weyl sense (see \cite{BS,GS} and references therein), and it is possible to associate to it a symbol function $\boldsymbol{\mathfrak{f}}$. More precisely, when $\boldsymbol{\mathfrak{f}}$ is smooth enough, if $N$ denotes the size of the adjacency matrix (i.e. the number of nodes of the graph), then the eigenvalues of the adjacency matrix are approximately values
of a uniform sampling of $\boldsymbol{\mathfrak{f}}$ in its definition domain, which
depends on $\Omega$ (see Definition \ref{def:eig-distribution} for the formal definition of eigenvalue distribution in the Weyl sense; the results in Section 5 of \cite{twin-theory} for the precise characterization of $\boldsymbol{\mathfrak{f}}$ and of its definition domain; and \cite{Barbarino-Bianchi-Garoni} for the definition of its monotone rearrangement).

 The knowledge of the symbol and of some of its basic analytical features provides a lot of information on the eigenvalue structure, of localization, spectral gap, clustering, and distribution type.

The mathematical tools are mainly taken from the field of Toeplitz structures (see \cite{BS} and \cite{GS,Tillinota,tyrtL1}) and of Generalized Locally Toeplitz (GLT) matrix-sequences (see  \cite{glt,glt-bis,Tilliloc}): for a recent account on the GLT theory and for several useful applications in the approximation of partial differential and fractional differential operators see \cite{Bianchi,BS18,DMS,GMS18,glt-book-1,glt-book-2,glt-book-3,glt-book-4,glt-book-3-old} and references therein.

Since in this paper we are interested in numerical applications of the theory developed in \cite{twin-theory}, we first show that many numerical schemes (see e.g. \cite{FE-book,IgA-book,FD}) for
 approximating partial differential equations (PDEs) and operators lead
 to sequences of structured matrices, which can be written as linear combination  of adjacency matrices, associated with sequences of graphs showing a uniform local structure. More specifically, if the physical domain of the differential
operator is $[0,1]^d$ (or any $d$-dimensional rectangle) and the coefficients are constant, then
we encounter $d$-level (weighted) Toeplitz graphs, when approximating the underlying PDE by using e.g. equispaced Finite Differences (FDs) or uniform Isogeometric Analysis (IgA). On the other hand, under the same assumptions on the underlying operator, quadrangular  and triangular Finite Elements (FEs) lead to block $d$-level Toeplitz structures, where the size of the blocks is related to the degree of the polynomial space of approximation and to the dimensionality $d$ (see \cite{FEM-paper}). Finally, in more generality, the GLT case is encountered by using any of the above numerical techniques, also with non-equispaced nodes/triangulations, when dealing either with a general domain $\Omega$ or when the coefficients of the differential operator are not constant. The given classification of approximated PDE matrix-sequences is relevant also because the obtained spectral information can be used for guiding the design of proper iterative solvers (in terms either of preconditioners or of ad hoc multigrid methods) for the underlying linear systems with large matrix size. In this direction, we provide a series of numerical examples where the knowledge of the symbol allows to obtain preconditioners for the conjugate gradient method, which significantly improve the computational cost for solving the involved large linear systems, and guides the choice of optimal projectors for two-grid and multigrid methods.

The paper is organized as follows. In Section \ref{sec:intro}  we will
first review some basic definitions and notation from graph theory,
from the field  of Toeplitz and $d$-level Toeplitz graphs, and we then provide the definitions of canonical spectral
distribution, graph Laplacian etc. We also briefly discuss some properties and tools useful in the study of multigrid methods. Section \ref{sec:num} will constitute the main core of our work. Here, a number of numerical applications are studied, highlighting the importance of the subject introduced in \cite{twin-theory} both for the discretization of PDEs and the study of elliptic problems in the discrete settings of weighted graphs. In particular, we numerically prove the relevance of the derived spectral information for building fast iterative solvers and we present various numerical examples to confirm the validity of our derivations. Finally, in Section \ref{sec:conclusion} we draw conclusions and discuss possible directions of research.

\section{Background, notation and classes of locally uniform graphs}\label{sec:intro}

In this section we recall some basics on graph theory (see, for example, the very recent \cite{Keller2021} for a modern exposition) and introduce definitions, notation and several families of graphs with uniform local structure, which will be of interest for our numerical applications. Before doing so, we briefly introduce a multi-index notation that will be used hereafter.\\

For a fixed integer $ d \geq 1 $, a $ d$-index $ \bfkk$ is a vector $ \bfkk=(k_1,...,k_d) \in \mathbb{Z}^d $. Given two $ d$-indices $ \bfii=(i_1,...,i_d) $ and $ \bfjj=(j_1,...,j_d) $, we write $\bfii \vartriangleleft \bfjj$ if $i_r<j_r$ for the first $r=1,2,\ldots,d$ such that $i_r\neq j_r$ and we write $\bfii < \bfjj$ if $i_r<j_r$ is satisfied for every $r=1,\ldots,d$. The relations $\trianglelefteq,\vartriangleright,\trianglerighteq$ and $\leq,>,\geq$ are defined accordingly.
	
We use bold letters for vectors and vector/matrix-valued functions. The notation $\bf{0},\bf{1},\bf{2},\ldots$ is used for the $d$-dimensional constant vectors $\left(0,0,\ldots,0\right)$, $\left(1,1,\ldots,1\right)$, $\left(2,2,\ldots,2\right),\ldots$, respectively. We intend every operation on vectors in $ \mathbb{Z}^d $ computed element-wise, so that, for example, the notation $\frac{\bfii}{\bfnn}$ is used for the vector $\frac{\bfii}{\bfnn}= \left(\frac{i_1}{n_1},\ldots,\frac{i_d}{n_d} \right)$ and $|\bfii|$ is used for the vector $|\bfii|=\left(|i_1|, \ldots, |i_d|\right)$. Finally, given a $d$-index $\bfnn$, we write $\bfnn \to \infty$ if and only if $\min_{r=1,\ldots,d}\{n_r\}\to \infty$.

\subsection{Graphs: basics, notation and $ d$-level diamond Toeplitz graphs}

We start this subsection by recalling some notation and basic information about graphs and subgraphs. After this, we introduce the main structures of graphs which will be of interest for numerical application, namely Toeplitz graphs, $d$-level Toeplitz graphs and $ d $-level diamond Toeplitz graphs (see \cite{twin-theory}).\\

A (possily infinite) \textit{graph} is a quadruple $G=(V,E,w, \kappa)$, where
\begin{itemize}
	\item  $V=\left\{ v_1, v_2,\ldots, v_n,... \right\}$ denotes a countable set of nodes, or vertices;
	\item  $E=\{(v_i,v_j)|\, v_i,v_j\in V, \; \exists \text{ an edge from } v_i \text{ to } v_j\}$ is the set of edges between nodes;
	\item $w: V\times V \to \R$ is the weight function and, given two nodes $ v_i $ and $ v_j $, we have $ w(v_i,v_j) \neq 0 $ if and only if there exists an edge from $ v_i $ to $ v_j $;
	\item  a potential term $\kappa :V \to \R$.
\end{itemize}
We use the notation $ G=(V,E,w) $ whenever the potential term $ \kappa $ vanishes.
The non-zero values of the weight function $w$ are called \textit{weights} associated with the edge $(v_i,v_j)$. A \textit{walk} of length $k$ in $G$ is a set of nodes $v_{i_0}, v_{i_1},\ldots,v_{i_{k-1}}, v_{i_{k}}$ such
that, for all $1\leq r\leq k$, $(v_{i_{r-1}},v_{i_{r}})\in E$. In this paper we work with \textit{undirected} and \textit{connected} graphs, i.e. graphs for which the function $ w $ is symmetric ($w(v_i,v_j)=w(v_j,v_i)$ for every $v_i,v_j  \in V$) and such that for every couple of vertices $v_i, v_j \in V $ there exists a walk from $ v_i $ to $ v_j $. If $ w(v_i,v_j) \neq 0 $ we write $ v_i \sim v_j $ and we say that $ v_i $ and $ v_j $ are \textit{neighbors}. Finally, the \textit{degree} of a node $v_i$ of an undirected graph, denoted by  $\deg(v_i)$, is defined as
\begin{equation*}
\deg(v_i):= \sum_{v_j\sim v_i}w(v_i,v_j).
\end{equation*}


We are now ready to introduce the main structures of graphs we are interested in for our numerical applications. We follow the same notation and the same increasing level of complexity proposed in \cite{twin-theory}, hence starting with the definition of Toeplitz graphs.\\

\begin{defn}[Toeplitz graph]\label{def:toeplitz-graph}
Let $n,m,t_1,\ldots,t_m$ be positive integers such that $0< t_1<t_2<\ldots<t_m< n$, and fix $m$ nonzero real numbers $w_{t_1},\ldots,w_{t_m}$. A \textnormal{Toeplitz graph} of cardinality $ n $ is the graph $T_n\langle (t_1,w_{t_1}),\ldots,(t_m,w_{t_m})\rangle=(V_n,E_n,w) $ such that $V_n=\{v_1,\ldots,v_n\}$ and the weight function $w$ satisfies
$$
w(v_i,v_j) = \begin{cases}
w_{t_k} & \mbox{if } |i-j|=t_k,\\
0 & \mbox{otherwise}.
\end{cases}
$$
\end{defn}

Note that, by construction, the adjacency matrix of a Toeplitz graph, that is the matrix $ W_n=(w(v_i,v_j))_{i,j=1}^n=(w_{|i-j|})_{i,j=1}^n $, is a symmetric Toeplitz matrix. Trivially, there exists a 1-1 correspondence between Toeplitz graphs and (symmetric) Toeplitz matrices.\\

A generalization of Toeplitz graphs is represented by $d$-level Toeplitz graphs. In order to introduce them, we first need to recall the concept of \textit{directions} associated with a $d$-index.\\

Given a $d$-index $\bftt_k=((t_k)_1,\ldots,(t_k)_d)$ such that $\bfzero \trianglelefteq \bftt_k$ and $\bftt_k \neq \bfzero$, the set
$$
[\bftt_k] := \quot{I_k}{\sim}, \qquad \mbox{where } \bfii \sim \bfjj \mbox{ iff } \bfii = \pm \bfjj,
$$
where
$$
I_k := \left\{  \bfii \in \mathbb{Z}^d\, | \,\bfii = \left(\pm (t_k)_1, \ldots, \pm (t_k)_d \right)  \right\},
$$
is called the set of \textit{directions} associated with $\bftt_k$. For $\alpha=1,\ldots,\left|[\bftt_k]\right|$, the elements $[\bftt_k]_{\alpha}\in[\bftt_k]$ are called \textit{directions} and clearly $\left|[\bftt_k]_\alpha\right|=2$. The element in $[\bftt_k]_\alpha$ whose first nonzero component is positive is denoted by $[\bftt_k]_\alpha^+$ and the other one is denoted by $[\bftt_k]_\alpha^-$.

\begin{defn}[$d$-level Toeplitz graph]\label{def:d-toeplitz-graph}
Let $\bfnn,\bftt_1,\ldots,\bftt_m$ be $d$-indices such that $\bf{0}<\bfnn$, $ 0\vartriangleleft\bftt_1\triangleleft \bftt_2 \triangleleft\ldots \triangleleft \bftt_m \triangleleft \bfnn-\bfuno, $
and let $\bfww_1,\ldots, \bfww_m$ be $m$ nonzero real vectors such that $\bfww_k \in \R^{c_k}$, with $c_k=\left|[\bftt_k]\right|$, for every $k=1,\ldots,m$. The components of the vectors $\bfww_k$ are denoted as follows:
$$
\bfww_k = \left( w_{[\bftt_k]_1}, w_{[\bftt_k]_2}, \ldots,  w_{[\bftt_k]_{c_k}} \right).
$$
A $d$-\textnormal{level
Toeplitz graph} $ T_\bfnn\langle \{[\bftt_1],\bfww_1\},\ldots,\{[\bftt_m], \bfww_m\}\rangle $ is an undirected graph with node set given by
$V_\bfnn=\left\{v_{\bfkk}\,|\, \bf{1}\trianglelefteq \bfkk\trianglelefteq\bfnn \right\}$ and whose weight function $\omega$ satisfies
\begin{equation}\label{eq:weight_dToeplitz}
w(v_\bfii,v_\bfjj) = \begin{cases}
w_{[\bftt_k]_\alpha} & \mbox{if } |\bfii-\bfjj|=\bftt_k \mbox{ and } (\bfii-\bfjj) \in [\bftt_k]_\alpha=\left\{ [\bftt_k]_\alpha^+, [\bftt_k]_\alpha^-\right\}\\ & \mbox{for some }\alpha=1,\ldots, c_k;\\
0 & \mbox{otherwise}.
\end{cases}
\end{equation}
\end{defn}

When every node of a $d$-level Toeplitz graph is replaced by a fixed, undirected graph of dimension $ \nu $, which we call \textit{diamond}, the graph assumes a block structure. To properly define this concept, recall that a \emph{linking graph operator} for the reference node set $ [\nu]:=\{1,\ldots,\nu\} $ is any non-zero $ \R^{\nu \times \nu} $ matrix, or, equivalently, the adjacency matrix of a (possibly not undirected) graph. We can then give the following definition.

\begin{defn}[$d$-level diamond Toeplitz graph]\label{def:d-level-diamond}
	Let $d,m,\nu$ be fixed integers and let $G\simeq \left([\nu], E,w\right)$ be a fixed undirected graph which we call \textnormal{mold graph}. 
	
	Let $\bfnn,\bftt_1,\ldots,\bftt_m$ be $d$-indices such that $\bf{0}<\bfnn$, and $0\vartriangleleft\bftt_1\triangleleft \bftt_2 \triangleleft\ldots \triangleleft \bftt_m \triangleleft \bfnn-\bfuno.$ For $k=1,\ldots,m$, let $\boldsymbol{L}_k$ be a collection of linking graph operators for the reference set $[\nu]$ such that $\left|\boldsymbol{L}_k \right|=c_k$, with $c_k=\left|[\bftt_k]\right|$ for every $k=1,\ldots,m$. We indicate the elements of the set $\boldsymbol{L}_k$ by the following index notation,
\begin{align*}
	&\boldsymbol{L}_k = \left\{ L_{[\bftt_k]_1}, L_{[\bftt_k]_2}, \ldots,  L_{[\bftt_k]_{c_k}} \right\},\\ 
	&\R^{\nu\times\nu}\ni L_{[\bftt_k]_\alpha}=\left(l_{[\bftt_k]_\alpha}(r,s)\right)_{r,s=1}^\nu \quad \mbox{for } \alpha=1,\ldots, c_k.
\end{align*}
Finally, consider $\bfnn$ copies $G(\bfkk)\simeq G$ of the mold graph, which we call \textnormal{diamonds}.
	
	A $d$-\textnormal{level diamond
		Toeplitz graph}, denoted by 
	$$
	T_{\bfnn,\nu}^G\left\langle \left\{\bftt_1,\boldsymbol{L}_1\right\},\ldots,\left\{\bftt_m,\boldsymbol{L}_m\right\} \right\rangle,
	$$
	 is an undirected graph with nodes
	$$
	V_\bfnn=\left\{v_{(\bfkk,r)}\,|\, ({\bf{1}},1)\trianglelefteq ({\bfkk}, r)\trianglelefteq (\bfnn,\nu) \right\}
	$$
	and weight function $w_{\bfnn}: V_\bfnn \times V_\bfnn\to \R$ satisfying
	$$
	w_{\bfnn}\left(v_{(\bfii,r)},v_{(\bfjj,s)}\right) := \begin{cases}
	w(r,s) & \mbox{if }\bfii=\bfjj,\\
	l_{[\bftt_k]_\alpha}(r,s) & \mbox{if } |\bfii-\bfjj|=\bftt_k \mbox{ and } (\bfii-\bfjj) = [\bftt_k]_\alpha^+,\\
	l_{[\bftt_k]_\alpha}(s,r) & \mbox{if } |\bfii-\bfjj|=\bftt_k \mbox{ and } (\bfii-\bfjj) = [\bftt_k]_\alpha^-,\\
	0 & \mbox{otherwise}.
	\end{cases}
	$$
\end{defn}

To end this section, we now introduce the main character of our numerical application, namely the graph Laplacian. First, observe that we use the notation $ C(V) $ to denote the set of real-valued functions on $ V $.

\begin{defn}[graph Laplacian]\label{def:graph Laplacian}
Let $G=(V,E,w,\kappa)$ be an undirected graph with no self-loops (i.e. $ w(v_i,v_i)=0 $ for every $ v_i \in V $). The \textnormal{graph Laplacian} is the symmetric matrix $\Delta_G : C(V) \to C(V)$ defined as
\begin{equation*}
\Delta_G := D+K - W,
\end{equation*}
where $D$ is the \textnormal{degree matrix}, $K$ is the \textnormal{potential term matrix}, that is,
$$
D:=\diag\left\{\deg(v_1),\ldots,\deg(v_n)\right\}, \quad K:=\diag\left\{\kappa(v_1),\ldots,\kappa(v_n)\right\},
$$ and $W$ is the adjacency matrix of the graph $G$, that is,
$$
W=\left(
\begin{array}{cccc}
0      &   w(v_1,v_2)  & \cdots & w(v_1,v_n)\\
w(v_1,v_2)    &    0   & \ddots & \vdots \\
\vdots & \ddots & \ddots & w(v_{n-1},v_n) \\
w(v_1,v_n) & \cdots & w(v_{n-1},v_n)     & 0\\
\end{array}
\right).
$$
Therefore, $ \Delta_G $ appears as follows:
\begin{equation*}
\Delta_G= \left(
\begin{array}{cccc}
\deg(v_1)  +\kappa(v_1)     &   -w(v_1,v_2)  & \cdots & -w(v_1,v_n)\\
-w(v_1,v_2)    &    \deg(v_2) + \kappa(v_2)   & \ddots & \vdots \\
\vdots & \ddots & \ddots & -w(v_{n-1},v_n) \\
-w(v_1,v_n) & \cdots & -w(v_{n-1},v_n)     & \deg(v_n)+\kappa(v_n)\\
\end{array}
\right).
\end{equation*}
\end{defn}

\subsection{Matrix-norms and partial ordering of Hermitian matrices}\label{sec:norms}

In this brief subsection, we introduce some notation for matrix-norms which will be employed in the following sections and we define the partial ordering on the subset of $ \mathbb{C}^{n \times n} $ formed by Hermitian matrices.\\
First, given a norm $ \| \cdot \| $ on $ \mathbb{C}^{n} $, it is always possible to consider the induced norm, which will be still denoted by $ \| \cdot \| $, on the space $ \mathbb{C}^{n \times n} $ defined as follows:
\begin{equation*}
\| A \| = \sup\left\{ \| Ax \|, \, x \in \mathbb{C}^n \text{ with } \| x \|=1 \right\} = \sup_{x \neq \bf{0}} \frac{\| Ax \|}{\|x\|}.
\end{equation*}
In our work we will be mainly interested in the case where the considered norm on $ \mathbb{C}^n $ is the p-norm, for some $ p \in [1, \infty] $, which we denote by $ \| \cdot \|_{p} $, if $ p \neq 2 $, and by $ \| \cdot \| $, if $ p =2 $.
In this last particular case the norm is also known as spectral norm and the following equality holds
\begin{equation*}
\| A \| = \sigma_{\max}(A),
\end{equation*}
where $ \sigma_{\max}(A) $ is the largest singular value of $ A $. 

We denote the Frobenius norm of a matrix $ A = (a_{i,j})_{i,j=1}^{n} $ by
\begin{equation*}
\| A \|_{F} = \sqrt{\sum_{i=1}^{n} \sum_{j=1}^{n} |a_{i,j} |^2}.
\end{equation*}
Let now $ A $ be a Hermitian Positive Definite (HPD) matrix. This assumption will be crucial in what follows, especially in Section \ref{sec:num}, for the definition of our iterative methods. In fact, in this special case, it is possible to define an inner product on $ \mathbb{C}^{n} $ as
\begin{equation*}
\left\langle u,v \right\rangle_A = u^T A v,
\end{equation*}
which induces the following norm on $ \mathbb{C}^{n} $
\begin{equation*}
\| u \|_A = \sqrt{\left\langle u,u \right\rangle_A} = \sqrt{u^T A u}.
\end{equation*}
To conclude this subsection, we define the following partial ordering on the subset of $ \mathbb{C}^{n \times n} $ of Hermitian matrices. Given two Hermitian matrices $ A $ and $ B $ we say that $ A \geq B $ (respectively, $ A > B $) if $ A - B $ is Hermitian Positive Semidefinite (HPSD) (respectively, $ A-B $ is HPD).

\subsection{Spectral symbol and generating function of $d$-level diamond Toeplitz graphs}\label{ssec:weyl_distribution}

A matrix-valued function
$\boldsymbol{\mathfrak{f}}:D\to\mathbb C^{\nu\times \nu}$, $\nu\geq 1$, defined on a measurable set
$D\subseteq\mathbb R^m$, $m\in \N$, is measurable (resp. continuous, in
$\textnormal{L}^p(D)$) if its components $\mathfrak{f}_{i,j}:D\to\mathbb C,\ i,j=1,\ldots,
\nu,$ are measurable (resp. continuous, in $\textnormal{L}^p(D)$). Let
$\mu_m$ be the Lebesgue measure on $\mathbb R^m$ and
let $C_c(\mathbb{R})$ be the set of continuous functions with
bounded support defined over $\mathbb{R}$.
We use the notation $ \{ X_{\bfnn,\nu} \} $ to denote a sequence of matrices of increasing dimension $ d_{\bfnn} $, i.e. such that $ d_{\bfnn} \to \infty$ as $ \bfnn \to \infty$, $ \nu $ being a fixed parameter independent of $ \bfnn$.



We say that $\{X_{\bfnn,\nu}\}_\bfnn$ is \textit{zero-distributed} if
\begin{equation}\label{zero-distr-sigma}
\lim_{\bfnn\rightarrow \infty}\frac{1}{d_\bfnn}\sum_{k=1}^{d_\bfnn}F(\sigma_k(X_{\bfnn,\nu})) = F(0) \qquad\forall F\in C_c(\mathbb{R}),
\end{equation}
and we indicate it by $\{X_{\bfnn,\nu}\}_\bfnn\sim_\sigma 0$, where $ \sigma_1(X_{\bfnn,\nu}),\ldots,\sigma_{d_{\bfnn}}(X_{\bfnn,\nu}) $ are the singular value of $ X_{\bfnn,\nu}$, sorted in non-decreasing order.

Of great importance for our numerical applications is the knowledge of the spectral symbol of a sequence of matrices, which we now recall.

\begin{defn}[Spectral symbol]\label{def:eig-distribution}
	Let $\{X_{\bfnn,\nu}\}_\bfnn$ be a sequence of matrices and let $\boldsymbol{\mathfrak{f}}:D\to\mathbb C^{\nu\times \nu}$ be a measurable Hermitian matrix-valued function defined on the measurable set
	$D\subset\mathbb R^m$, with $0<\mu_m(D)<\infty$.
	
	 We	say that $\{X_{\bfnn,\nu}\}_\bfnn$ is distributed like $\boldsymbol{\mathfrak{f}}$ in the sense of
	eigenvalues, in symbols $\{X_{\bfnn,\nu}\}_\bfnn \sim_\lambda \boldsymbol{\mathfrak{f}}$, if
	\begin{equation}\label{distribution:sv-eig}
	\lim_{\bfnn\rightarrow \infty}\frac{1}{d_\bfnn}\sum_{k=1}^{d_\bfnn}F(\lambda_k(X_{\bfnn,\nu}))=\frac1{\mu_m(D)}\int_D \sum_{k=1}^\nu  F\left(\lambda_k(\boldsymbol{\mathfrak{f}}(\boldsymbol{y}))\right)d\mu_m(\boldsymbol{y}),\qquad\forall F\in C_c(\mathbb{R}),
	\end{equation}
	where $\lambda_1(\boldsymbol{\mathfrak{f}}(\boldsymbol{y})), \ldots,\lambda_\nu(\boldsymbol{\mathfrak{f}}(\boldsymbol{y}))$ are the eigenvalues of $\boldsymbol{\mathfrak{f}}(\boldsymbol{y})$ and $ \lambda_1(X_{\bfnn,\nu}),\ldots,\lambda_{d_{\bfnn}}(X_{\bfnn,\nu}) $ are the eigenvalues of $ \{X_{\bfnn,\nu} \}$, sorted in non-decreasing order. 
\end{defn}

It is easy to show that (see \cite[Proposition 4.4]{twin-theory}), given a $ d $-level diamond Toeplitz graph \\$T_{\bfnn,\nu}^G\left\langle \left\{\bftt_1,\boldsymbol{L}_1\right\},\ldots,\left\{\bftt_m,\boldsymbol{L}_m\right\} \right\rangle$, there exists a function $\bfff : [-\pi,\pi]^d \to \mathbb{C}^{\nu\times \nu}$, called the \textit{generating function}, such that the adjacency matrix $ W^G_{\bfnn,\nu} $ of the graph satisfies
$$
(W^G_{\bfnn,\nu})_{\bfii,\bfjj}= \hat{\bfff}_{\bfii-\bfjj}
$$
and, therefore (see \cite{Tillinota}), $ \{ W^G_{\bfnn,\nu} \} \sim_{\lambda} \boldsymbol{\mathfrak{f}} \equiv \bfff $, whenever $ \bfff \in \text{L}^1([-\pi,\pi]^d) $.\\

\subsection{Multigrid Methods for Structured Matrices}\label{ssec:multigrid}

In the current subsection we summarize the key results on symbol-based multigrid methods for Toeplitz structures. Throughout the subsection, we consider the case of an invertible matrix $A_n \in \mathbb{C}^{d_n \times d_n}$ and a vector $b \in \mathbb{C}^{d_n}$.

Multigrid methods (MGM) are iterative procedures for large linear systems that create a proper sequence of linear systems of decreasing dimensions obtained by consecutive projections. First, we focus on a general coefficient matrix $A_{n}\in\mathbb{C}^{d_n \times d_n}$, then we treat in detail the results related to Toeplitz matrices $T_n(f)$.

Given $x_{n},\,b_{n}\in\mathbb{C}^{d_n}$, consider a full-rank matrix $P_{n,k}\in\mathbb{C}^{d_n\times k}$, $k<d_n$,  and let us consider two stationary iterative methods: the method $\mathcal{V}_{n,\rm{pre}}$, with iteration matrix ${V}_{n,\rm{pre}}$, and $\mathcal{V}_{n,\rm{post}}$, with iteration matrix ${V}_{n,\rm{post}}$.

Given an initial guess $x_{n}^{(0)}\in\mathbb{C}^{d_n}$, an iteration of a two-grid method (TGM) is given by the following steps:
\begin{center}
\begin{tabular}{lcl}
     \multicolumn{3}{c} {$x_{n}^{(k+1)}=\mathcal{TGM}(A_n,x_{n}^{(j)},b_n)$} \\ 
     \hline \\
    \fbox{\begin{tabular}{l} 0. $x_{n}^{\rm{pre}}=\mathcal{V}_{n,\rm{pre}}^{\nu_{\rm{pre}}}(A_{n},{b}_{n},x_{n}^{(j)})$ \end{tabular} }  && {Pre-smoothing iterations}\\
    \ \\
   \fbox{\begin{tabular}{l}
          1. $r_{n}=b_{n}-A_{n}x_{n}^{\rm{pre}}$ \\
          2. $r_{k}=P_{n,k}^{H}r_{n}$     \\
          3. $A_{k}=P_{n,k}^{H}A_{n}P_{n,k}$  \\
          4. Solve $A_{k}y_{k}=r_{k}$  \\
          5. $\hat{x}_{n}=x_{n}^{\rm{pre}}+P_{n,k}y_{k}$\\
    \end{tabular}
    } && {Coarse Grid Correction {(CGC)}}\\
    \ \\
    \fbox{\begin{tabular}{l}6. $x_{n}^{(j+1)}=\mathcal{V}_{n,\rm{post}}^{\nu_{\rm{post}}}(A_{n},{b}_{n},\hat{x}_{n})$\end{tabular}} && {Post-smoothing iterations} \\
\end{tabular}
\end{center}

The Coarse Grid Correction (CGC) depends on the grid transfer operator $P_{n,k}$, while step $0.$ and
step $6.$ consist, respectively, in applying $\nu_{\rm{pre}}$ times a pre-smoother
and $\nu_{\rm{post}}$ times a post-smoother of the given iterative methods.

We base the convergence analysis on the Ruge-St\"uben theory  \cite{RStub} for TGM. In particular, we mention the results in \cite[Remark~2.2]{ADS} and \cite[Theorem 5.2]{RStub}, which provide two separate conditions that must be fulfilled in order to have TGM convergence, namely the ``smoothing property'' and the ``approximation property''. We highlight that, if such conditions are satisfied, then the resulting TGM is not only convergent but has also an optimal convergence rate.

The standard V-cycle method is obtained by replacing the direct solution at step 4. with a recursive call of the TGM applied to the coarser linear system $A_{k_{\ell}}y_{k_{\ell}}=r_{k_{\ell}}$, where $\ell$ represents the level. The recursion is usually stopped at level ${\overline{\ell}}$ when $k_{\overline{\ell}}$ becomes small enough for solving cheaply step 4. with a direct solver.

In the case where $A_n=T_n(f)$, an efficient choice of grid transfer operator is given by $P_{n,k}$ obtained as the product between a Toeplitz matrix $T_{n}(p)$, with $p$ trigonometric polynomial, and the cutting matrix $K_{n}$. That is,
\begin{equation}\label{eq:def_projector_pnk}
P_{n,k}=T_{n}(p)K_{n},
\end{equation}
where, if the size $n$ of the coefficient matrix is divisible by a factor $\mathfrak{g}\ge2$, the lower level dimension is reduced to $k=n/\mathfrak{g}$ by the cutting matrix
\begin{equation}\label{eq:def_cutting_matrix}
K_{n}=[\delta_{i-\mathfrak{g}j}]_{i,j}, \quad i = 0, \dots, n -1; \, j = 0,\dots, k - 1,\quad \quad \delta_{\ell}=\begin{cases}
1 \, &{\rm if}\, \ell \equiv 0\, ({\rm mod}\, n),\\
0 \, &{\rm otherwise}
\end{cases}.
\end{equation}

The validation of the approximation property depends on the choice of the generating function $p$, which has to fulfill specific conditions depending on the zeros of $f$.

In the case where $f$ is a $d$-variate function, i.e. $T_\bfnn(f)$ is a $d$-level matrix,  we construct $K_{\bfnn}$ as the tensor product of the cutting matrices $K_{n_1}$, $K_{n_2}$, $\dots$, $K_{n_d}$.

In what follows we provide, in the $d$-level setting, conditions on $p$ in order to obtain a projector which is effective in terms of convergence. For the generalization to the case where $\bfff$ is a $\nu\times \nu$ matrix-valued generating function see \cite{multi-block}.

\begin{thm}[\cite{ADS,DSS,Fiorentino-Serra}]\label{thm:toep_approx_prop}
Consider a Toeplitz matrix $T_\bfnn(f)$  generated by a non-negative trigonometric polynomial $f$. Suppose that $f(\bftheta)$ vanishes at exactly one point $\mathbf{\bftheta_0}$. Then, the optimality of the two-grid method applied to $T_\bfnn(f)$ is guaranteed if we choose a projection operator of the form (\ref{eq:def_projector_pnk}) associated with a generating function $p$ such that
   \begin{equation*}\label{eqn:mgm_approx_cond}
\begin{split}
\underset{ {\bftheta}\to  {\bftheta_0}}{\lim\sup}\frac{|p( {\bfeta})|^2}{f( {\bftheta})}&<\infty, \quad  {\bfeta}\in \mathcal{M}_\mathfrak{g}( {\bftheta}),\\
\sum_{ {\bfeta}\in \Omega( {\bftheta})}p^2( {\bfeta})&>0,
\end{split}
 \end{equation*}
where 
$$\Omega_\mathfrak{g}( {\bftheta})=\left\{\bfeta\left|\eta_j\in\left\{\theta_j+\frac{2\pi k}{\mathfrak{g}}(\mbox{\rm mod } 2\pi)\right\},k=0,\dots,\mathfrak{g}-1,j=1,\dots,d\right.\right\}$$
 and $\mathcal{M}_\mathfrak{g}( {\bftheta})=\Omega_\mathfrak{g}( {\bftheta})\setminus \{\bftheta\}$ are the sets of {corner} and {mirror} points respectively.

%


\end{thm}

For achieving the optimality of the V-cycle method, the first condition needs to be strengthened, see \cite{ADS} for details. Moreover, it was proven in \cite{DSS} that the coarsening factor $\mathfrak{g}$ needs to be strictly greater than 2 for an overall optimal method, otherwise each multigrid iteration would require a computational cost greater than the cost of the matrix vector product. Finally, we highlight that a coarsening factor greater than 2 permits to construct less coarser levels before solving the error equation exactly. This implies that the critical issues (i.e. low rank corrections, increase of the condition number), that could worsen level after level, are propagated a smaller number of times.

\section{Numerical examples}\label{sec:num}

In this section we provide numerical examples which emphasize the computational help of the theory developed in \cite{twin-theory} for studying many different kinds of discretization of PDEs as well as for studying elliptic problems in the very general setting of weighted graphs. Using the spectral features of the considered graphs and in particular the analytical properties of the underlying symbol, we can define appropriate preconditioners for the preconditioned conjugate gradient (PCG) (see \cite{saad} for some general theory on preconditioned Krylov methods) and appropriate projectors for the multigrid method (MGM) (see \cite{Oost}), with the aim of making such methods optimal. For the optimality of MGM for unilevel Toeplitz structures corresponding to approximated differential equations in one dimension see \cite{Fiorentino-Serra}; for multi-level Toeplitz and non-Toeplitz structures corresponding to approximated PDEs in dimension greater than one see \cite{ADS,Sun}. Furthermore, we show that without the proposed preconditioning suggested in the present work, the convergence to the solution of the problem can be very slow, due to the bad conditioning of the discrete  problems under consideration.\\

First, we show some numerical experiments which exploit the complete power of the graph structures recalled in Section \ref{sec:intro} and the related spectral properties developed in \cite{twin-theory}. At this purpose, we first perform tests in which graphs are used as a discretization of the geometry of a sub-domain of $[0,1]^d$, $ d \geq 1 $. Then, we move into a complete graph setting, where the considered graph structure can represent a model for various type of networks and real-world interactions. For the sake of completeness with respect to the first part of the present work, we also perform numerical tests concerning examples in \cite[Subsections 7.2 and 7.3]{twin-theory}, which have been already investigated in the relevant literature (see, for example, \cite{sp-gr,DBFS93,solvers15,solvers-bis15,our-sinum}).\\

Before proceeding, we recall some notation concerning (proper) subgraphs and the concept of immersion of a $d$-level Toeplitz graph in a subdomain of $ [0,1]^{d} $.\\

Given a graph $\bar{G}=(\bar{V},\bar{E},\bar{w},\bar{\kappa})$ and a (proper) subset $V\subset \bar{V}$, we say that $ G=(V,E,w,\kappa) $ is a (proper) subgraph of $ \bar{G}$, and we write $ G \subset \bar{G}$ if
\begin{itemize}
	\item $V\subset \bar{V}$;
	\item $E=\{(v_i,v_j) \in \bar{E} \, | \, v_i,v_j \in V \}\subset \bar{E}$;
	\item $w=\bar{w}_{|E}$;
	\item $\kappa = \bar{\kappa}_{|\mathring{V}}$,
\end{itemize}
where the set of nodes
$$
\mathring{V}:=\left\{ v_i \in V \, | \, v_i \nsim \bar{v}_j \, \forall \, \bar{v}_j \in \bar{V}\setminus V \right\}
$$
is called \textit{interior} of $ V $ and its element are called \textit{interior nodes}.\\
Finally, given a subgraph $G=(V,E,w,\kappa)$ of $ \bar{G}=(\bar{V},\bar{E},\bar{w},\bar{\kappa})$, we can define the \textit{boundary} of $ V $ as
$$
\partial V:=\left\{ v_i \in V \, | \, v_i \sim \bar{v}_j \, \mbox{for some } \, \bar{v}_j \in \bar{V}\setminus V \right\},
$$
whose elements are called  \textit{boundary nodes}. Note that we do not request that $\kappa =\bar{\kappa}$ on $\partial V$.\\

For the concept of immersion, consider a $d$-level Toeplitz graph $ T_\bfnn\langle
\{[\bftt_1],\bfww_1\},\ldots,\{[\bftt_m],\bfww_m\}\rangle $ with weight function $ w $ and a continuous almost everywhere function $ p:[0,1]^{d} \to \mathbb{R} $. We then consider the immersion map  $\iota : V_\bfnn \to (0,1)^d$ such that
$$
\iota(v_\bfjj):= \bfjj \circ \bfhh = \left(j_1h_1,\ldots, j_dh_d \right),
$$
where $ \circ $ denotes the Hadamard (component-wise) product and $ \bfhh $ is the $ d$-dimensional vector
$$
		\bfhh:=(h_1,\ldots,h_d)=\left(\frac{1}{n_1+1},\ldots, \frac{1}{n_d+1}\right).
$$
The map $ \iota $ induces a grid graph $ G=(V'_{\bfnn},E'_{\bfnn}, w^p) $ in $ [0,1]^d $, where
$$
V'_\bfnn:=\left\{ \bfxx_\bfkk = \iota(v_\bfkk)\,|\, \bf{1}\trianglelefteq \bfkk \trianglelefteq\bfnn  \right\}, \qquad E'_\bfnn :=\left\{ (\bfxx_\bfii,\bfxx_\bfjj) \, | \, w^p(\bfxx_\bfii,\bfxx_\bfjj)\neq 0  \right\},
$$
$$
w^p (\bfxx_\bfii,\bfxx_\bfjj) := 
p\left(\frac{\bfxx_\bfii + \bfxx_\bfjj}{2}\right)w(v_\bfii,v_\bfjj).
$$
To extend this concept to any subdomain $ \Omega$ of $ [0,1]^d $ we simply restrict the set of nodes $ V'_{\bfnn} $ to
$$
V_{n'}^{\Omega} := V'_{\bfnn} \cap \Omega
$$
and consequently modify the weight function $ w^p $ to
$$
w^{\Omega,p}:=w^p_{|V_{n'}^{\Omega} \times V_{n'}^{\Omega}}.
$$

It is clearly possible to further extend the same idea to the case of $d$-level diamond Toeplitz graphs. In this case, however, the choice of the immersion map $ \iota $ is not as natural as in the previous case and different choices of the immersion map $\iota $ would be able to describe different grid geometries.

\subsection{Toeplitz graph with Fourier coefficients immersed in the triangle}\label{ssec:multigrid_triangle}
In this first numerical example our starting model operator is the classic (semi-positive definite) Laplacian on a triangular domain with Dirichlet or Neumann boundary conditions (BCs), that is 
\begin{align}
	&\mathcal{L}:\textnormal{dom}(\mathcal{L})\subset L^2\left(D\right) \to L^2\left(D\right), \label{eq:SLO1tr}\\
	&\mathcal{L}u(x,y):= -\left(\partial^2_{xx}u(x,y)+\partial^2_{yy}u(x,y)\right) \qquad (x,y)\in D,\label{eq:SLO2tr} 
\end{align}
where $D$ is an equilateral triangle contained in $[0,1]^2$. As usual, $L^2\left(D\right)$ is the space of squares integrable functions and $\textnormal{dom}(\mathcal{L})$ is an appropriate (Sobolev) subspace of $L^2\left(D\right)$ such that the formal equation \eqref{eq:SLO2tr} is well-defined and the BCs are satisfied. For example, in the case of Dirichlet BCs, $\textnormal{dom}(\mathcal{L})= \overline{C_c^\infty(D)}$ is the closure of the space of smooth and compactly supported functions in $D$ with respect to the Sobolev norm. For this numerical test we do not apply a standard discretization scheme.

The idea is to build a full graph-based discretization scheme looking at the operator itself rather than at the specific numerical approximation method, by using the Fourier coefficients of $h(\theta)=\theta^2$ defined on $[0,\pi]$. This is due to the fact that, as it was observed in \cite[Corollary C.1]{Bianchi}, when discretizing the one-dimensional Laplace operator, the sequence of spectral symbols associated to the uniform FD scheme with $(2\eta+1)$-points converges uniformly to $h$ as $\eta\to \infty$. This phenomenon is not restrained to the FD discretization scheme, but it appears in the IgA framework (\cite[Theorem 1, Theorem 2 and Lemma 1]{EFGMSS}) and in the Sinc collocation method as well (\cite{bookColl}), once the class of regularity of the approximating functions is progressively refined. Let us observe that $h$ is exactly the (inverse and normalized) asymptotic distribution function, in the Weyl sense, of the one-dimensional Laplace operator on the unit interval. In some sense, all those methods become ``indistinguishable'' in the limit of their regularity parameters, from a spectral point of view, since their sequences of spectral symbols converge to the same function, that is, the distribution function of the operator they are approximating, as one should expect. Therefore, our approach is to exploit directly $h$ as a weight function.  

As a matter of facts, $h(\theta)=\theta^2$ is the asymptotic distribution function of the one-dimensional Laplacian while here we are considering the two-dimensional Laplacian. This is not an issue, since the space of smooth functions of the form $g(x,y)=\sum_{k=1}^n \alpha_k g_{1,k}(x)g_{2,k}(y)$ is dense in $L^2(D)$ (for a reference see \cite{BD}, for example). Therefore, the two-dimensional Laplacian can be approximated by ``splitting'' it into two one-dimensional Laplacians that act separately along the two different axes.

The full graph-based discretization we implement here is then the following: first, fix $\bfnn = (n,n)$ and consider the infinite graph $\bar{G}_\bfnn$ given by
\begin{equation*}\label{eq:triangle-graph}
	\bar{G}_{\bfnn}=\left(\bar{T}_{\bfnn}\left\langle \cup_{k=1}^\infty \left\{[k,0],w_k  \right\},\cup_{k=1}^\infty \left\{[0,k],w_k  \right\}\right\rangle\right),
\end{equation*}
where $\bar{T}_{\bfnn}\left\langle \cup_{k=1}^\infty \left\{[k,0],w_k  \right\},\cup_{k=1}^\infty \left\{[0,k],w_k  \right\}\right\rangle$ is a $2$-level infinite Toeplitz graph characterized by the node set
\begin{equation*}\label{eq:def_Vn_tr}
	\bar{V}_{\bfnn}= \left\{ (x_i,y_j) =  \left(\frac{i}{n+1}, \frac{j}{n+1}\right)\, : \, i,j \in \mathbb{Z} \right\}, 
\end{equation*}
and the weight function 
\begin{equation*}
	w((x_i,y_j),(x_r,y_s)):= w_k= (-1)^{k}\frac{2}{k^2} \quad \mbox{if } (|i-r|,|j-s|)\in\{(k,0),(0,k)\}.
\end{equation*}
We highlight that $(-1)^{k+1}\frac{2}{k^2}$ is the $k$-th term of the Fourier expansion of $h(\theta)=\theta^2$ on $[0,\pi]$. Finally, consider the subgraph $G_\bfnn \subset \bar{G}_{\bfnn}$ such that 
$$
G_{\bfnn}=\left(T^D_{\bfnn}\left\langle \cup_{k=1}^n \left\{[k,0],w_k  \right\},\cup_{k=1}^n \left\{[0,k],w_k  \right\}\right\rangle, \kappa \right),
$$
and where the node set of $T^D_{\bfnn}$ is given by
$$
V_\bfnn = \bar{V}_{\bfnn} \cap D.
$$ 
The potential term $\kappa$ determines the BCs. If $\kappa =0$ then we call it the \emph{Neumann potential}: this is related to the fact that in $\R^d$, the Neumann BCs that characterizes a second-order elliptic differential operator, as \eqref{eq:SLO1tr}-\eqref{eq:SLO2tr}, disappears when
passing to the closure of its associated quadratic form, see \cite[Theorem 7.2.1]{Davies}. Instead, if
$$
k((x_i,y_j))= \sum_{(x_r,y_s) \in \bar{V}_\bfnn\setminus V_\bfnn} w((x_i,y_j),(x_r,y_s)),
$$
then we call it \emph{Dirichlet potential}. It takes into account the edge deficiency of each node in $G_\bfnn$ when seen as a node in the host graph $\bar{G}_\bfnn$. The Dirichlet potential term arises naturally for example when discretizing a differential operator with Dirichlet BCs by means of FD schemes, see \cite[pp. 40--42]{Bianchi}. For more details about Dirichlet potentials we refer to \cite[Section 2.2]{Keller2021} and \cite[Definition 2.1 and Lemma A.1]{BSW}. Look at Figure \ref{fig:immersion_in_the_triangle} for a visual representation of the graph. 

\begin{figure}
	\begin{center}
		\resizebox{0.8\textwidth}{!}{\begin{tikzpicture}[whitestyle/.style={circle,draw,fill=white!40,minimum size=4},
				graystyle/.style ={ circle ,top color =white , bottom color = gray ,
					draw,black, minimum size =4},	state2/.style ={ circle ,top color =white , bottom color = gray ,
					draw,black , text=black , minimum width =1 cm},
				state3/.style ={ circle ,top color =white , bottom color = white ,
					draw,white , text=white , minimum width =1 cm},
				state/.style={circle ,top color =white , bottom color = white,
					draw,black , text=black , minimum width =1 cm}]
				\draw[step=0.5] (-3,-3) grid (3,3);
				\draw[step=0.5, ForestGreen,very thick] (-2.5,-2.5) grid (2.5,2.5);
			   \node (t0) at (-3,-3) {}; 
			   \node (t1) at (3,-3) {}; 
			   \node (t2) at (0,2.1962) {}; 
			   \begin{scope}
			   	\pgfsetfillpattern{mynewdots}{ForestGreen}
			   	\fill (t0.center)--(t1.center)--(t2.center);
			   \end{scope}
			   \begin{scope}
			   	\pgfsetfillpattern{crosshatch dots}{red}
			   	\fill (-3,-3)--(3,-3)--(3,3)--(-3,3);
			   \end{scope}
			   \path[draw,very thick] (t0.center)--(t1.center);
			   \path[draw,very thick] (t1.center)--(t2.center);
			   \path[draw,very thick] (t2.center)--(t0.center);
				\draw[very thick, blue] (2,-2) ellipse (0.9cm and 0.9cm);
				\draw[blue, very thick, ->] (2.9,-2) to (4,-0.5);

				\foreach \x in {-2.5,-2,-1.5,-1,-0.5,0,0.5,1,1.5,2,2.5}
				\draw[red,very thick] (\x,2) to (\x,2.5);
				\foreach \x in {-2.5,-2,-1.5,-1,-0.5,0.5,1,1.5,2,2.5}
				\draw[red,very thick] (\x,1.5) to (\x,2);
				\foreach \x in {-2.5,-2,-1.5,-1,-0.5,0.5,1,1.5,2,2.5}
				\draw[red,very thick] (\x,1) to (\x,1.5);
				\foreach \x in {-2.5,-2,-1.5,-1,1,1.5,2,2.5}
				\draw[red,very thick] (\x,0.5) to (\x,1);
				\foreach \x in {-2.5,-2,-1.5,-1,1,1.5,2,2.5}
				\draw[red,very thick] (\x,0) to (\x,0.5);
				\foreach \x in {-2.5,-2,-1.5,1.5,2,2.5}
				\draw[red,very thick] (\x,-0.5) to (\x,0);
				\foreach \x in {-2.5,-2,2,2.5}
				\draw[red,very thick] (\x,-1) to (\x,-0.5);
				\foreach \x in {-2.5,-2,2,2.5}
				\draw[red,very thick] (\x,-1.5) to (\x,-1);
				\foreach \x in {-2.5,2.5}
				\draw[red,very thick] (\x,-2) to (\x,-1.5);
				\foreach \x in {-2.5,2.5}
				\draw[red,very thick] (\x,-2.5) to (\x,-2);
				
				\foreach \y in {-2,-1.5,-1,-0.5,0,0.5,1,1.5,2,2.5}
				\draw[red,very thick] (-2.5,\y) to (-2,\y);
				\foreach \y in {-2,-1.5,-1,-0.5,0,0.5,1,1.5,2,2.5}
				\draw[red,very thick] (2,\y) to (2.5,\y);
				
				\foreach \y in {-1,-0.5,0,0.5,1,1.5,2,2.5}
				\draw[red,very thick] (-2,\y) to (-1.5,\y);
				\foreach \y in {-1,-0.5,0,0.5,1,1.5,2,2.5}
				\draw[red,very thick] (1.5,\y) to (2,\y);				
				
				\foreach \y in {0,0.5,1,1.5,2,2.5}
				\draw[red,very thick] (-1.5,\y) to (-1,\y);
				\foreach \y in {0,0.5,1,1.5,2,2.5}
				\draw[red,very thick] (1,\y) to (1.5,\y);		
				
				\foreach \y in {0.5,1,1.5,2,2.5}
				\draw[red,very thick] (-1,\y) to (-0.5,\y);
				\foreach \y in {0.5,1,1.5,2,2.5}
				\draw[red,very thick] (0.5,\y) to (1,\y);	
				
				\foreach \y in {1.5,2,2.5}
				\draw[red,very thick] (-0.5,\y) to (0,\y);
				\foreach \y in {1.5,2,2.5}
				\draw[red,very thick] (0,\y) to (0.5,\y);				
				
				\foreach \x in {-2.5,-2,-1.5,-1,-0.5,0,0.5,1,1.5,2,2.5}
				\foreach \y in {-2.5,-2,-1.5,-1,-0.5,0,0.5,1,1.5,2,2.5}
				\node [whitestyle] at (\x,\y) {};
				
				
				\foreach \x in {-2.5,-2,-1.5,-1,-0.5,0,0.5,1,1.5,2,2.5}
				\foreach \y in {2.5}
				\node [graystyle] at (\x,\y) {};			
				
				\foreach \x in {-2.5,-2,-1.5,-1,-0.5,0.5,1,1.5,2,2.5}
				\foreach \y in {1.5,2}
				\node [graystyle] at (\x,\y) {};
									
				\foreach \x in {-2.5,-2,-1.5,-1,1,1.5,2,2.5}
				\foreach \y in {1}
				\node [graystyle] at (\x,\y) {};

				\foreach \x in {-2.5,-2,-1.5,-1,1,1.5,2,2.5}
				\foreach \y in {0.5}
				\node [graystyle] at (\x,\y) {};
				
				\foreach \x in {-2.5,-2,-1.5,1.5,2,2.5}
				\foreach \y in {0}
				\node [graystyle] at (\x,\y) {};
				
				\foreach \x in {-2.5,-2,2,2.5}
				\foreach \y in {-1,-0.5}
				\node [graystyle] at (\x,\y) {};
				
				\foreach \x in {-2.5,2.5}
				\foreach \y in {-2,-1.5}
				\node [graystyle] at (\x,\y) {};
							
				\node[state] at (4.5,2) (A) {$\kappa$};
				\node[state]  (J) [right=of A] {$\kappa$};
				\node[state2]  (B) [right=of J] {};
				\node[state]  (D) [below=of A] {$\kappa$};
				\node[state]  (K) [right=of D] {$\kappa$};
				\node[state2]  (E) [right=of K] {};
				\node[state]  (G) [below=of D] {$\kappa$};
				\node[state]  (L) [right=of G] {$\kappa$};
				\node[state]  (H) [right=of L] {$\kappa$};
				\path[-] (A) edge [red, bend left=35] node[above right,scale=0.8] {$\textcolor{red}{w_2}$} (B);
				\path[-] (G) edge [ForestGreen, bend right=35] node[above right,scale=0.8] {$\textcolor{ForestGreen}{w_2}$} (H);
				\path[-] (D) edge [red, bend left=35] node[sloped,above left,scale=0.8] {$\textcolor{red}{w_2}$} (E);
				\path[-] (A) edge [ForestGreen, bend right=45] node[sloped,below right,scale=0.8] {$\textcolor{ForestGreen}{w_2}$} (G);
				\path[-] (J) edge [ForestGreen, bend left=35] node[sloped,above right,scale=0.8] {$\textcolor{ForestGreen}{w_2}$} (L);
				\path[-] (B) edge [red, bend left=45] node[sloped,above right,scale=0.8] {$\textcolor{red}{w_2}$} (H);
				
				\path[-] (A) edge [ForestGreen] node[right,scale=0.8] {$\textcolor{ForestGreen}{w_1}$} (D);
				\path[-] (J) edge [ForestGreen] node[right,scale=0.8] {$\textcolor{ForestGreen}{w_1}$} (K);
				\path[-] (K) edge [ForestGreen] node[right,scale=0.8] {$\textcolor{ForestGreen}{w_1}$} (L);
				\path[-] (A) edge [ForestGreen] node[above,scale=0.8] {$\textcolor{ForestGreen}{w_1}$} (J);
				\path[-] (D) edge [ForestGreen] node[above right,scale=0.8] {$\textcolor{ForestGreen}{w_1}$} (K);
				\path[-] (G) edge [ForestGreen] node[above,scale=0.8] {$\textcolor{ForestGreen}{w_1}$} (L);
				\path[-] (J) edge [red] node[above,scale=0.8] {$\textcolor{red}{w_1}$} (B);
				\path[-] (K) edge [red] node[above right,scale=0.8] {$\textcolor{red}{w_1}$} (E);
				\path[-] (L) edge [ForestGreen] node[above,scale=0.8] {$\textcolor{ForestGreen}{w_1}$} (H);
				\path[-] (B) edge [red] node[right,scale=0.8] {$\textcolor{red}{w_1}$} (E);
				\path[-] (H) edge [red] node[right,scale=0.8] {$\textcolor{red}{w_1}$} (E);
				\path[-] (D) edge [ForestGreen] node[right,scale=0.8] {$\textcolor{ForestGreen}{w_1}$} (G);
				
				\draw[red,dashed, bend right=45] (E) to (8.7,3);
				\draw[red,dashed, bend right=35] (H) to (8.9,3);
				\draw[red,dashed] (B) to (8.5,3);
				\draw[ForestGreen,dashed, bend right=45] (J) to (3.5,2.2);
				\draw[red,dashed, bend right=45] (B) to (3.5,2.4);
				\draw[ForestGreen,dashed, bend left=45] (K) to (3.5,-0.2);
				\draw[red,dashed, bend left=45] (E) to (3.5,-0.4);
				\draw[ForestGreen,dashed, bend left=45] (L) to (3.5,-2.2);
				\draw[ForestGreen,dashed, bend left=35] (H) to (3.5,-2.4);
				\draw[ForestGreen,dashed] (A) to (4.5,3);
				\draw[ForestGreen,dashed, bend left=45] (D) to (4.3,3);
				\draw[ForestGreen,dashed, bend left=35] (G) to (4.1,3);
				\draw[ForestGreen,dashed] (J) to (6.5,3);
				\draw[red,dashed, bend left=45] (K) to (6.3,3);
				\draw[red,dashed, bend left=35] (L) to (6.1,3);
				\draw[ForestGreen,dashed] (A) to (3.5,2);
				\draw[ForestGreen,dashed] (D) to (3.5,0);
				\draw[ForestGreen,dashed] (G) to (3.5,-2);
		\end{tikzpicture}}
	\end{center}\caption{Immersion of a $2$-level grid graph inside the equilateral triangle $D$. The white nodes are the nodes of $V_\bfnn$ while the gray nodes belong to $\bar{V}_\bfnn\setminus V_\bfnn$. The green connections represent the weighted edges whose end-nodes are both interior nodes of $V_\bfnn$, while the red connections represent the weighted edges which have at least one end-node that belongs to $\bar{V}_\bfnn\setminus V_\bfnn$. For the seek of clarity, in the whole figure on the left we explicitly draw with continuous lines only the edges of distance 1, while a complete representation of the edges can be seen in the enlargement on the right.}\label{fig:immersion_in_the_triangle}
\end{figure}
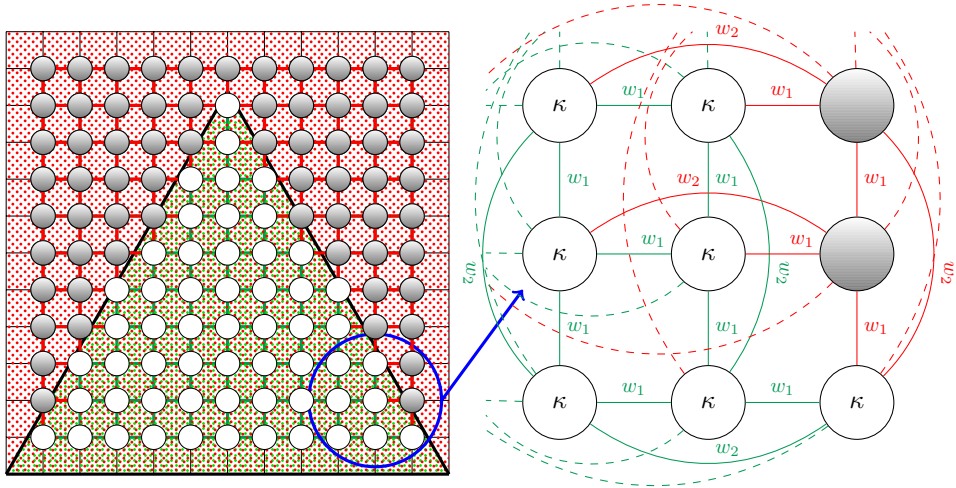

In the end, we are going to solve the equation
\begin{equation}\label{eq:triangle-linearsys}
	\Delta_{G_\bfnn} u_\bfnn = b_\bfnn, 
\end{equation}
where $\Delta_{G_\bfnn}$ is the graph Laplacian associated to $G_\bfnn$. Let us observe that $(n+1)^2\Delta_{G_\bfnn}$ approximates $\mathcal{L}$, see Table \ref{table:ex_triangle_eig} for a comparison of the spectra of the discrete and continuous operators.

\begin{table}
	\centering
	\begin{tabular}[m]{c|ccc} 
		\toprule
		$\floor*{\frac{k(\bfnn)}{d_\bfnn}}$   & $n=16$     & $n=32$  & $n=64$   \\ 
		\midrule
		0.1 & 0.0350    &   0.0429   & 0.0153    \\
		0.5                            &  0.1177    &   0.0711    & 0.0200    \\
		0.8                           &    0.1493    &   0.0821     & 0.0196    \\
		\bottomrule 
	\end{tabular}
	\begin{tabular}[m]{cccc} 
		 &  &  &    \\ 
	\end{tabular}
	\begin{tabular}[m]{c|ccc} 
		\toprule
		$\floor*{\frac{k(\bfnn)}{d_\bfnn}}$   & $n=16$     & $n=32$  & $n=64$   \\ 
		\midrule
		0.1 & 0.0536    &   0.0277    & 0.0153    \\
		0.5                            &  0.0529    &   0.0297    & 0.0200    \\
		0.8                           &    0.0819    &   0.0334    & 0.0196    \\
		\bottomrule
	\end{tabular}
	\caption{Relative errors between the eigenvalue $\lambda_{k(\bfnn)}^{(\bfnn)}$ and the corresponding evaluation of the continuous Laplacian eigenvalue on the equilateral triangle, with Neumann (left) and Dirichlet (right) BCs. The index $k(\bfnn)$ is chosen such that $k(\bfnn)/d_\bfnn$ is constant for every fixed $\bfnn=(n,n)$ with $n=16, 32, 64$.}\label{table:ex_triangle_eig}
\end{table}

\begin{rem}
	Let us observe that here we allow negative values for the weight function $w$. This is not uncommon, since in some applications it is natural to admit graphs with possible negative values associated to their edges. Nevertheless, if we restrain $w$ to be nonnegative then we can recover again the same operator $\Delta_{G_\bfnn}$ as sum of path graph Laplacians (\cite{Estrada}). For such a construction we refer to \cite[Sections II and III]{Bianchi-Donatelli}.
\end{rem}
We have the following preliminary result.
\begin{prop}\label{prop:eig_distr_triangle}
	Given the sequence of graph Laplacians $\{\Delta_{G_\bfnn}\}$ defined above, it holds that
	\begin{equation*}
		\{\Delta_{G_\bfnn}\} \sim_\lambda \mathfrak{f}(\theta_1,\theta_2)= \theta_1^2 + \theta_2^2, \quad (\theta_1,\theta_2) \in [0,\pi]^2.
	\end{equation*}
\end{prop}
\begin{proof}
To ease the notation of the proof, we write $v_r = (x_i,y_j)$ and sort them by lexicographic ordering. Let us consider the graph 
$$
F_{\bfnn}=\left(T^Q_{\bfnn}\left\langle \cup_{k=1}^n \left\{[k,0],w_k  \right\},\cup_{k=1}^n \left\{[0,k],w_k  \right\}\right\rangle, \kappa \right), \qquad Q=[0,1]^2,
$$
where the node set of $T^D_{\bfnn}$ is given by
$$
V^Q_\bfnn = \bar{V}_{\bfnn} \cap Q.
$$ 
Clearly, it holds that $G_\bfnn \subset F_\bfnn \subset \bar{G}_\bfnn$. Consider now the graph Laplacian associated to $F_\bfnn$ with the Dirichlet potential. We have that $\Delta_{F_\bfnn} \in \R^{n^2\times n^2}$ and, by Definition \ref{def:graph Laplacian}, each element of the main diagonal is of the form
\begin{align*}
\left(\Delta_{F_\bfnn}\right)_{i,i} &= \deg(v_r) + k(v_r) \\
&= \sum_{\substack{v_s\sim v_r\\ v_s \in V^Q_\bfnn}} w(v_r,v_s)  + \sum_{\substack{v_s\sim v_r\\ v_s \in \bar{V}_\bfnn \setminus V^Q_\bfnn}} w(v_r,v_s) \\
&= \sum_{\substack{v_s\sim v_r\\ v_s \in \bar{V}_\bfnn}} w(v_r,v_s)\\
&= 2 \sum_{k \in \mathbb{Z}} (-1)^{k}\frac{2}{k^2}\\
&= \frac{2\pi^2}{3}.
\end{align*}
Noticing that $\left\{\frac{2\pi^2}{3} \right\}\cup \left\{ \left((-1)^k\frac{2}{k^2}, (-1)^j\frac{2}{j^2}\right) \right\}_{(k,j) \in \mathbb{Z}\times \mathbb{Z} \setminus (0,0)}$ are the Fourier coefficients of $\mathfrak{f}(\theta_1, \theta_2)= \theta_1^2 + \theta_2^2$ on $[0,\pi]^2$, and that $\Delta_{F_\bfnn}$ is a $2$-level Toeplitz matrix whose diagonals of each diagonal blocks are given exactly by the Fourier coefficients of $\mathfrak{f}$, that is, $\left(\Delta_{F_\bfnn}\right)_{\bfii,\bfjj}:=\left(\hat{\mathfrak{f}}_{\bfii-\bfjj}\right)_{\bfii,\bfjj=\bfuno}^{\bfnn}$, then by standard theory it follows easily that
$$
\{\Delta_{F_\bfnn}\} \sim_\lambda \mathfrak{f}.
$$  
Consider now the sequence of matrices $ \{ I_{\bfnn}(\chi_D) \}_{\bfnn} $ such that 
$$
I_{\bfnn}(\chi_D) := \diag\left( \chi_D \left( \frac{\bfii}{\bfnn+\bf{1}} \right) \right)_{\bfii=\bf{1},\ldots,\bfnn}.
$$
It is known (see for example \cite[Lemma 2.3]{Barbarino}) that
$$
\{ I_{\bfnn(\chi_D)} \}_{\bfnn} \sim_{\lambda} \chi_D.
$$
Moreover, it is not difficult to see that the sequence $ \{ I_{\bfnn} \}_{\bfnn} $ belongs to the wide family of GLT sequences with symbol $ \chi_D $. Using the property of $*$-algebra of the GLT sequences (see \cite[GLT properties, pg. 118-119]{glt-book-2}) it is immediate to conclude that the sequence $ \{ I_{\bfnn}(\chi_D) \Delta_{F_{\bfnn}} I_{\bfnn}(\chi_D) \} $ satisfies
$$
\{ I_{\bfnn}(\chi_D) \Delta_{F_{\bfnn}} I_{\bfnn}(\chi_D) \} \sim_{\lambda} \mathfrak{f} \, \chi_{D}.
$$
Now, $ \Delta_{G_\bfnn} $ is a principal sub-matrix of $ \Delta_{F_\bfnn} $ (and of $ I_{\bfnn}(\chi_D) \Delta_{F_{\bfnn}} I_{\bfnn}(\chi_D) $) for every $ \bfnn $; therefore, in light of \cite[Lemma 3.4]{Barbarino}, we can conclude that 
$$
\{ \Delta_{G_\bfnn} \} \sim_{\lambda} \mathfrak{f}_{|D}.
$$
For the case $ \kappa = 0 $ the only difference is that now the main diagonal is no longer constant. Nevertheless, the main diagonal differs from that of $ \Delta_{F_{\bfnn}} $ of a term of small norm and this, therefore, does not affect the resulting symbol function. Proceeding as above we have our conclusion.

\end{proof}

Let us now investigate the convergence rate of the discrete solution $u_\bfnn$ to the continuous solution $u$. In the following we consider the Dirichlet BCs case, but an analogous study can be done for Neumann BCs. For this purpose, we consider the function $u$ given by the formula
\[
u(x,y)=y(y - \sqrt{3}x)(y + \sqrt{3}x - \sqrt{3})
\]
which vanishes on the boundary of the triangle $D$ and satisfies the equation $\mathcal{L}[u](x,y)=2\sqrt{3}$. We compute the evaluations $u(x_i,y_j)$ for all $(x_i,y_j)\in V_\bfnn$ and we sort them into the vector $u^*_\bfnn$ in lexicographic ordering. Then, we solve the linear system (\ref{eq:triangle-linearsys}) with right-hand side $b_\bfnn$ equal to the vector of all $2\sqrt{3}$. In Table \ref{table:ex_triangle_error} and Figure \ref{fig:ex_triangle_error} we report the relative error of the discrete solution $u_\bfnn$ with respect to evaluations in $u^*_\bfnn$ and we see that it decreases proportionally to the partial dimension $n$.

\begin{table}[H]
	\begin{minipage}{0.5\textwidth}
		\centering
	\begin{tabular}[m]{ccc} 
		\toprule
		$n$ & $d_\bfnn$ & $\frac{\|u_\bfnn-u_\bfnn^*\|}{\|u_\bfnn^*\|}$ \\
		\midrule
		6   & 18    & 0.3157 \\
		14  & 90    & 0.1131 \\
		30  & 400   & 0.0638 \\
		62  & 1686  & 0.0287 \\
		126 & 6920  & 0.0146 \\
		254 & 28028 & 0.0071 \\
		\bottomrule
	\end{tabular}\caption{The relative error of the approximation vector $u_\bfnn$, solution of (\ref{eq:triangle-linearsys}), with respect to the vector of evaluations $u_\bfnn^*$ of the exact solution increasing $n$ in the Dirichlet BCs case.}\label{table:ex_triangle_error}
\end{minipage}\hfill
\begin{minipage}[m]{0.45\textwidth}
	\centering
	\includegraphics[height=50mm]{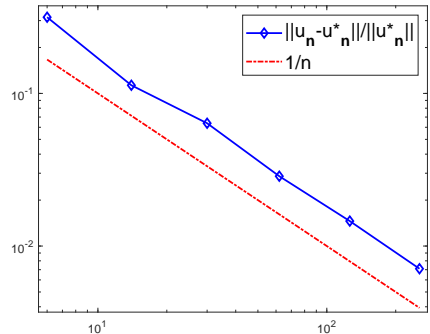}  
	\captionof{figure}{The log-log scale plot of the 2-norm relative error increasing $n$.}
		\label{fig:ex_triangle_error}
	\end{minipage}
\end{table}

In the following, we propose a multigrid strategy to solve the linear system (\ref{eq:triangle-linearsys}) based on the spectral distribution provided by Proposition \ref{prop:eig_distr_triangle}, taking for simplicity Gauss-Seidel both as a pre-smoother and post-smoother.

Assume that the partial dimension is of the form $n=2^t$, $t\in\mathbb{N}$. The graph Laplacian $\Delta_{G_\bfnn}$ has size $d_\bfnn$ such that $d_\bfnn< n^2$. Since the spectral symbol $\mathfrak{f}(\theta_1,\theta_2)= \theta_1^2 + \theta_2^2$ has a zero of order 2 in $(0,0)$, according to Theorem \ref{thm:toep_approx_prop} we choose the trigonometric polynomial $q(\theta)=4+6\cos(\theta)+4\cos(2\theta)+2\cos(3\theta)$ to construct the grid transfer operator
\[
P_{\bfnn,\bfkk}=T_\bfnn\left((q(\theta_1))(q(\theta_2))\right)K_\bfnn
\]
with $\bfkk=\left(\frac{n}{\mathfrak{g}},\frac{n}{\mathfrak{g}}\right)$, where $\mathfrak{g}$ represents the coarsening factor. Then, we eliminate from the matrix $P_{\bfnn,\bfkk}$ the rows $\mathbf{i}=(i_1,i_2)$ such that $(x_{i_1},y_{i_2})$ does not belong to $V_{\bfnn}$ and eliminate the columns $\mathbf{j}=(j_1,j_2)$ such that $(x_{j_1},y_{j_2})$ does not belong to $V_{\bfkk}$.

In Table \ref{table:iter_triangle_dir} we numerically show the validity of our proposed methods, reporting the number of iterations needed for achieving the tolerance $\varepsilon = 10^{-6}$ when increasing the matrix size. In the first and second columns, we show the results relative to the two-grid and V-cycle methods for $\mathfrak{g}=2$, while in the third and fourth columns we consider a more aggressive coarsening with $\mathfrak{g}=4$. For $\mathfrak{g}=2$, we see that the number of iterations needed for convergence remains almost constant, when increasing the size $d_\bfnn$. The method with $\mathfrak{g}=4$ has a worse convergence rate, but it can be convenient when considering greater matix-sizes from the computational complexity point of view, as explained in Subsection \ref{ssec:multigrid}.

\begin{table}[htb]
	\begin{center}
		\begin{tabular}{cccccccccc}
			\toprule
			{$t$}&{$d_\bfnn$}& {Two-grid ($\mathfrak{g}=2$)} & {V-cycle ($\mathfrak{g}=2$)} & {Two-grid ($\mathfrak{g}=4$)}& {V-cycle ($\mathfrak{g}=4)$}\\
			\midrule
			3  & 30    & 9  & 9  & 25 & 25\\
			4  & 116   & 10 & 10 & 27 & 27\\
			5  & 454   & 10 & 11 & 33 & 33\\
			6  & 1796  & 10 & 11 & 36 & 37\\
			7  & 6920  & 11 & 12 & 38 & 40 \\
			8  & 28028 & 11 & 12 & 39 & 41 \\
			\bottomrule
		\end{tabular}
	\end{center}
	\caption{Comparison of the number of iterations, varying $t$, for the two-grid and V-cycle methods applied to the linear system in (\ref{eq:triangle-linearsys}) in the Dirichlet BCs case.}
	\label{table:iter_triangle_dir}
\end{table}

In the Neumann BCs case, the matrix $\Delta_{G_\bfnn}$ is singular, since the the column vector $e$ of all ones of dimension $d_\bfnn$ is an eigenvector for $\Delta_{G_\bfnn}$ associated to the null eigenvalue. Then, we consider the matrix $A_\bfnn=\Delta_{G_\bfnn}+\frac{1}{d_\bfnn}ee^T$ and we solve the system $A_\bfnn y_\bfnn=b_\bfnn$. We have
\begin{equation}\label{eq:equivalent_system}
b_\bfnn=A_\bfnn y_\bfnn=(\Delta_{G_\bfnn}+\frac{1}{d_\bfnn}ee^T)y_\bfnn=\Delta_{G_\bfnn}y_\bfnn+\frac{1}{d_\bfnn}e(e^Ty_\bfnn)
\end{equation}
and the latter equality splits $b_\bfnn$ into the sum of a term in the range of $\Delta_{G_\bfnn}$ and a term in the null space of $\Delta_{G_\bfnn}$. Since the vector $b_\bfnn$ belongs to the range of $\Delta_{G_\bfnn}$, then $ee^Ty_\bfnn$ is the vector of all zeros and $y_\bfnn$ is the solution also of system (\ref{eq:triangle-linearsys}).

Note that the matrix $ee^T$ is a circulant matrix, hence the addition does not worsen the matrix-vector product computational cost. Since the matrix $A_\bfnn$ is SPD, we can apply the PCG for solving the linear system $A_\bfnn y_\bfnn=b_\bfnn$. In Table \ref{table:iter_triangle_neu}, we compare the efficiency, in terms of iteration count increasing the matrix-size, of two different preconditioning strategies. We see from the first iteration column that the number of iterations needed by the CG method with no preconditioning for reaching the desired tolerance $\epsilon=10^{-6}$ increases significantly. The first preconditioning strategy that we use is inspired by circulant preconditioning. Indeed, we construct the Strang circulant precondioner (plus a rank 1 correction) and we remove the rows and columns that correspond to the nodes that do not belong to the considered equilateral triangle. This procedure is far from giving an optimal convergence rate. The last strategy that we try consists in applying an MGM as preconditioner. In particular, we choose a V-cycle method with fixed tolerance $10^{-1}$ which uses Gauss-Seidel as pre-smoother and post-smoother and the standard linear interpolation as grid transfer operator. The results are shown in the last column of Table \ref{table:iter_triangle_neu}, where the number of iterations needed for the MGM-PCG to reach the tolerance $\epsilon=10^{-6}$ is almost constant as $d_\bfnn$ increases.

\begin{table}[htb]
	\begin{center}
		\begin{tabular}{cccccccccc}
			\toprule
			{$t$}&{$d_\bfnn$}& {CG} & {Circulant--PCG} & {MGM--PCG}\\
			\midrule
			3  & 30    & 30 & 21 & 6\\
			4  & 116   & 67 & 30 & 8\\
			5  & 454   & $>100$ & 42 & 9\\
			6  & 1796  & $>100$ & 60 & 9\\
			\bottomrule
		\end{tabular}
	\end{center}
	\caption{Comparison of the number of iterations, varying $t$, for the PGC method with no preconditioning, a circulant preconditioner and the MGM preconditioner applied to the linear system $A_\bfnn y_\bfnn=b_\bfnn$ in the Neumann BCs case.}
	\label{table:iter_triangle_neu}
\end{table}

\subsection{A multigrid method for a Toeplitz graph immersed in the disk}\label{ssec:multigrid_disk}
In this subsection we develop and study multigrid methods for the example in \cite[Subsection 7.1]{twin-theory}. For the reader convenience, we briefly summarize hereafter the mentioned example. 

We consider the operator $\mathcal{L}$,
\begin{align}
	&\mathcal{L}:W^{1,2}_0\left(B_{1/2}\right) \to L^2\left(B_{1/2}\right), \label{eq:SLO1}\\
	&\mathcal{L}u(x,y):= -\textnormal{div}\left[p(x,y)\nabla u(x,y)\right] + q(x,y)u(x,y) \qquad (x,y)\in B_{\frac{1}{2}},\label{eq:SLO2}
\end{align}
where
$$
B_{1/2}=\left\{(x,y)\in \R^2\, : \, 4\left(x-\frac{1}{2}\right)^2+4\left(y-\frac{1}{2}\right)^2<1\right\},
$$
and $W^{1,2}_0\left(B_{1/2}\right)$ is the closure of the space of smooth and compactly supported functions in the disk with respect to the Sobolev norm. $\mathcal{L}$ is characterized by Dirichlet boundary conditions on the boundary of the disk $B_{1/2}\subset[0,1]^2$. Fixing the diffusion term $p(x,y)=1+(x-1/2)^2+(y-1/2)^2$ and the potential term $q(x,y)=\textrm{e}^{xy}$, the discretization is made by an equispaced two-dimensional central FD approximation with $5$-points. Fix $\bfnn = (n, n)$. Then, the resulting graph $G_{\bfnn}$ that approximates the underlying geometry is given by
\begin{equation*}\label{eq:5-FD-graph}
	G_{\bfnn}=\left(T_{\bfnn}^{B_{1/2}}\left\langle \left\{[1,0],w^p  \right\},\left\{[0,1],w^p  \right\}\right\rangle, \kappa\right),
\end{equation*}
where $T_{\bfnn}^{B_{1/2}}\left\langle \left\{[1,0],w^p  \right\},\left\{[0,1],w^p  \right\}\right\rangle$ is a $2$-level Toeplitz graph characterized by the node set
\begin{equation*}\label{eq:def_Vn}
	V^{B_{1/2}}_{\bfnn}= \left\{(x_i,y_j) \in [0,1]^2 \, : \, (x_i,y_j) \in B_{\frac{1}{2}} \right\}, \quad (x_i,y_j)= \left(\frac{i}{n+1}, \frac{j}{n+1}\right)\mbox{ for } i,j=1,\ldots, n,
\end{equation*}
and the weight function 
\begin{equation*}
	w^p((x_i,y_j),(x_r,y_s)):=  p\left(\frac{x_i+x_r}{2},  \frac{y_j+y_s}{2} \right) \quad \mbox{if } (|i-r|,|j-s|)\in\{(1,0),(0,1)\}.
\end{equation*}
The potential term $\kappa$ is the Dirichlet potential, as we already introduced it in the previous Subsection \ref{ssec:multigrid_triangle}. In particular, $G_{\bfnn}$ is a sub-graph of a graph $\bar{G}_{\bfnn}$, of the same structure of $G_{\bfnn}$, but whose node set $\bar{V}_{\bfnn}$ is given by all the points $(x_i,y_j)$ in $[0,1]^2$ and whose weight function $\bar{w}^p$ is obtained by extending $p$ continuously outside the disk, fixing $p(x,y)\equiv 5/2$ for every $(x,y)\in [0,1]^2\setminus B_{1/2}$. We have that
\begin{equation*}
	\kappa(x_i,y_j)=\begin{cases}
		\kappa_0(x_i,y_j)= h^2q(x_i,y_j) &\mbox{if } (x_i,y_j) \in \mathring{V}^{B_{1/2}}_\bfnn,\\
		\kappa_1(x_i,y_j)= h^2q(x_i,y_j) + \frac{5}{2}, &\mbox{if } \exists! \mbox{ one neighbor}\\ & \mbox{in } \bar{V}_\bfnn\setminus V^{B_{1/2}}_\bfnn,\\
		\kappa_2(x_i,y_j)= h^2q(x_i,y_j) + \frac{10}{2}, &\mbox{if } \exists! \mbox{ two neighbors}\\ & \mbox{in } \bar{V}_\bfnn\setminus V^{B_{1/2}}_\bfnn.\\
	\end{cases}
\end{equation*}
For all the details we refer to \cite[Subsection 7.1]{twin-theory}. See Figure \ref{fig:immersion_in_the_disk}.

\begin{figure}
	\begin{center}
		\resizebox{0.8\textwidth}{!}{\begin{tikzpicture}[whitestyle/.style={circle,draw,fill=white!40,minimum size=4},
				graystyle/.style ={ circle ,top color =white , bottom color = gray ,
					draw,black, minimum size =4},	state2/.style ={ circle ,top color =white , bottom color = gray ,
					draw,black , text=black , minimum width =1 cm},
				state3/.style ={ circle ,top color =white , bottom color = white ,
					draw,white , text=white , minimum width =1 cm},
				state/.style={circle ,top color =white , bottom color = white,
					draw,black , text=black , minimum width =1 cm}]
				\draw[step=0.5] (-3,-3) grid (3,3);
				\draw[step=0.5, red,very thick] (-2.5,-2.5) grid (2.5,2.5);
				\draw[step=0.5, ForestGreen,very thick] (-2,-2) grid (2,2);
				\draw[step=0.5, ForestGreen, very thick] (-1.5,-2.5) grid (1.5,2.5);
				\draw[step=0.5, ForestGreen, very thick] (-2.5,-1.5) grid (2.5,1.5);
				\draw[very thick] (0,0) circle (3);
				\draw[very thick, blue] (2.25,-2) ellipse (0.4cm and 0.8cm);
				\draw[blue, very thick, ->] (2.65,-2) to (4,-0.5);
				\draw[very thick, blue] (-2.5,0) ellipse (0.4cm and 0.8cm);

				\foreach \y in {-1.5,...,1.5}
				\draw[red,very thick] (2.5,\y) to (3,\y);
				\foreach \y in {-1,...,1}
				\draw[red,very thick] (2.5,\y) to (3,\y);
				
				\foreach \y in {-1.5,...,1.5}
				\draw[red,very thick] (-2.5,\y) to (-3,\y);
				\foreach \y in {-1,...,1}
				\draw[red,very thick] (-2.5,\y) to (-3,\y);
				
				\foreach \x in {-1.5,...,1.5}
				\draw[red,very thick] (\x,-2.5) to (\x,-3);
				\foreach \x in {-1,...,1}
				\draw[red,very thick] (\x,-2.5) to (\x,-3);
				
				\foreach \x in {-1.5,...,1.5}
				\draw[red,very thick] (\x,2.5) to (\x,3);
				\foreach \x in {-1,...,1}
				\draw[red,very thick] (\x,2.5) to (\x,3);

				\foreach \x in {-1,...,1}
				\node [whitestyle] at (\x,2.5) {};
				\foreach \x in {-0.5,...,0.5}
				\node [whitestyle] at (\x,2.5) {};
				
				\foreach \x in {-2,...,2}
				\node [whitestyle] at (\x,2) {};
				\foreach \x in {-1.5,...,1.5}
				\node [whitestyle] at (\x,2) {};
				
				\foreach \x in {-2,...,2}
				\node [whitestyle] at (\x,1.5) {};
				\foreach \x in {-1.5,...,1.5}
				\node [whitestyle] at (\x,1.5) {};
				
				\foreach \x in {-2.5,...,2.5}
				\node [whitestyle] at (\x,1) {};
				\foreach \x in {-2,...,2}
				\node [whitestyle] at (\x,1) {};
				
				\foreach \x in {-2.5,...,2.5}
				\node [whitestyle] at (\x,0.5) {};
				\foreach \x in {-2,...,2}
				\node [whitestyle] at (\x,0.5) {};
				
				\foreach \x in {-2.5,...,2.5}
				\node [whitestyle] at (\x,0) {};
				\foreach \x in {-2,...,2}
				\node [whitestyle] at (\x,0) {};
				
				\foreach \x in {-2.5,...,2.5}
				\node [whitestyle] at (\x,-0.5) {};
				\foreach \x in {-2,...,2}
				\node [whitestyle] at (\x,-0.5) {};
				
				\foreach \x in {-2.5,...,2.5}
				\node [whitestyle] at (\x,-1) {};
				\foreach \x in {-2,...,2}
				\node [whitestyle] at (\x,-1) {};
				
				\foreach \x in {-2,...,2}
				\node [whitestyle] at (\x,-1.5) {};
				\foreach \x in {-1.5,...,1.5}
				\node [whitestyle] at (\x,-1.5) {};
				
				\foreach \x in {-2,...,2}
				\node [whitestyle] at (\x,-2) {};
				\foreach \x in {-1.5,...,1.5}
				\node [whitestyle] at (\x,-2) {};
				
				\foreach \x in {-1,...,1}
				\node [whitestyle] at (\x,-2.5) {};
				\foreach \x in {-0.5,...,0.5}
				\node [whitestyle] at (\x,-2.5) {};
				
				\node [whitestyle] at (1.5,2.5) {};
				\node [whitestyle] at (1.5,-2.5) {};
				\node [whitestyle] at (-1.5,2.5) {};
				\node [whitestyle] at (-1.5,-2.5) {};
				
				\node [whitestyle] at (-2.5,-1.5) {};
				\node [whitestyle] at (2.5,1.5) {};
				\node [whitestyle] at (-2.5,1.5) {};
				\node [whitestyle] at (2.5,-1.5) {};
				\node [graystyle] at (2.5,2.5) {};
				\node [graystyle] at (2,2.5) {};
				\node [graystyle] at (2.5,2) {};
				
				\node [graystyle] at (-2.5,-2.5) {};
				\node [graystyle] at (-2,-2.5) {};
				\node [graystyle] at (-2.5,-2) {};
				
				\node [graystyle] at (-2.5,2.5) {};
				\node [graystyle] at (-2,2.5) {};
				\node [graystyle] at (-2.5,2) {};
				
				\node [graystyle] at (2.5,-2.5) {};
				\node [graystyle] at (2,-2.5) {};
				\node [graystyle] at (2.5,-2) {};
				\node[state] at (5,2) (A) {$\kappa_0$};
				\node[state]  (B) [right=of A] {$\kappa_2$};
				\node[state]  (D) [below=of A] {$\kappa_2$};
				\node[state2]  (E) [right=of D] {};
				\node[state2]  (G) [below=of D] {};
				\node[state2]  (H) [right=of G] {};
				\path[-] (A) edge [ForestGreen] node[above] {$\textcolor{black}{w^p}$} (B);
				\path[-] (G) edge [red] node[above] {$\textcolor{black}{w^{\bar{p}}}$} (H);
				\path[-] (D) edge [red] node[above] {$\textcolor{black}{w^{\bar{p}}}$} (E);
				\path[-] (H) edge [red] node[right] {$\textcolor{black}{w^{\bar{p}}}$} (E);
				
				\path[-] (A) edge [ForestGreen] node[right] {$\textcolor{black}{w^p}$} (D);
				\path[-] (B) edge [red] node[right] {$\textcolor{black}{w^{\bar{p}}}$} (E);
				\path[-] (D) edge [red] node[right] {$\textcolor{black}{w^{\bar{p}}}$} (G);
				
				\draw[ForestGreen,dashed] (B) to (7,3);
				\draw[red,dashed] (B) to (8,2);
				\draw[ForestGreen,dashed] (A) to (5,3);
				\draw[ForestGreen,dashed] (A) to (4,2);
				\draw[ForestGreen,dashed] (D) to (4,0);
				\node[state] at (-5,0) (A1) {$\kappa_1$};
				\node[state]  (B1) [above=of A1] {$\kappa_1$};
				\node[state]  (C1) [below=of A1] {$\kappa_1$};
				\node[state2]  (D1) [left=of A1] {};
				\node[state2]  (E1) [left=of B1] {};
				\node[state2]  (F1) [left=of C1] {};
				\draw[blue, very thick, ->] (-2.9,0) to (-3.8,0);
				\path[-] (A1) edge [ForestGreen] node[right] {$\textcolor{black}{w^p}$} (B1);
				\path[-] (A1) edge [ForestGreen] node[right] {$\textcolor{black}{w^p}$} (C1);
				\path[-] (A1) edge [red] node[below] {$\textcolor{black}{w^{\bar{p}}}$} (D1);
				\path[-] (B1) edge [red] node[below] {$\textcolor{black}{w^{\bar{p}}}$} (E1);
				\path[-] (C1) edge [red] node[below] {$\textcolor{black}{w^{\bar{p}}}$} (F1);
				\path[-] (D1) edge [red] node[right] {$\textcolor{black}{w^{\bar{p}}}$} (F1);
				\path[-] (D1) edge [red] node[right] {$\textcolor{black}{w^{\bar{p}}}$} (E1);

				\draw[ForestGreen,dashed] (B1) to (-5,3);
				\draw[ForestGreen,dashed] (C1) to (-5,-3);
				\draw[ForestGreen,dashed] (C1) to (-4,-2);
				\draw[ForestGreen,dashed] (A1) to (-4,0);
				\draw[ForestGreen,dashed] (B1) to (-4,2);
		\end{tikzpicture}}
	\end{center}\caption{Immersion of a $2$-level grid graph inside the disk $B_{1/2}\subset [0,1]^2$. The white nodes are the nodes of $V^{B_{1/2}}_\bfnn$ while the gray nodes belong to $\bar{V}_\bfnn\setminus V^{B_{1/2}}_\bfnn$.  The potential term of the host graph $\bar{G}_n$ is determined only by the potential term $q$ from \eqref{eq:SLO1} while the potential term $\kappa$ of the sub-graph $G_\bfnn$ is influenced by the nodes in $\bar{V}_\bfnn\setminus V^{B_{1/2}}_\bfnn$ on the boundary set $\partial V^{B_{1/2}}_\bfnn$. This influence is due to the presence of Dirichlet BCs in \eqref{eq:SLO2}. The green connections represent the weighted edges whose end-nodes are both interior nodes of $V^{B_{1/2}}_\bfnn$, while the red connections represent the weighted edges which have at least one end-node that belongs to $\bar{V}_\bfnn\setminus V^{B_{1/2}}_\bfnn$. The potential term $\kappa$ sums the weight of a red edge to every of its end-nodes which belong to $V^{B_{1/2}}_\bfnn$.}\label{fig:immersion_in_the_disk}
\end{figure}
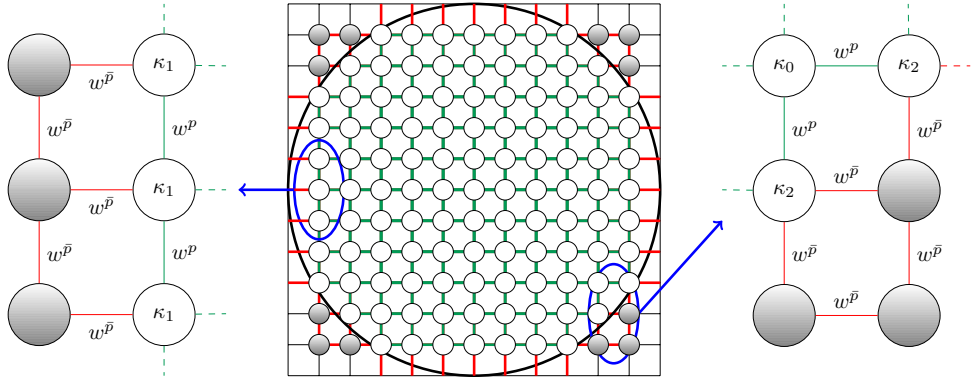

The graph Laplacian $\Delta_{G_\bfnn}$, associated to the graph $G_\bfnn$, approximates the normalized operator $(n+1)^{-2}\mathcal{L}$. By  \cite[Corollary 5.3]{twin-theory} it immediately follows that
\begin{equation*}
	\{ \Delta_{G_\bfnn} \} \sim_{\lambda} \mathfrak{f}(x,y,\theta_1,\theta_2), \qquad (x,y,\theta_1,\theta_2) \in B_{1/2}\times[-\pi,\pi]^2
\end{equation*}
where
\begin{align}\label{eq:distribution_FD_example}
	\mathfrak{f}(x,y,\theta_1,\theta_2)&=p(x,y)\left(4-f(\theta_1,\theta_2)\right) \nonumber\\
	&=\left[1+(x-1/2)^2+(y-1/2)^2\right]\left(4 - 2\cos(\theta_1)  -2\cos(\theta_2)\right),
\end{align}
and because of a symmetric argument, we can restrict $\mathfrak{f}$ on $B_{1/2}\times [0,\pi]^2$ without affecting the validity of the identity \eqref{distribution:sv-eig}. Let us define the function $g:[-\pi,\pi]^2\rightarrow \mathbb{R}$
\[
g(\theta_1,\theta_2)=4 - 2\cos(\theta_1)-2\cos(\theta_2)
\]
and the function $a:[0,1]^2\rightarrow \mathbb{R}$ as the composition of the potential term $p(x,y)=1+(x-1/2)^2+(y-1/2)^2$ with the following mapping of the square $[0,1]^2$ into the disk $B_{1/2}$
\[
m(x,y)=\left(x\sqrt{1-\frac{1}{2}y^2},y\sqrt{1-\frac{1}{2}x^2}\right),
\]
that is,
\[
a(x,y)=p(m(x,y))=1+\left(x\sqrt{1-\frac{1}{2}y^2}-\frac{1}{2}\right)^2+\left(y\sqrt{1-\frac{1}{2}x^2}-\frac{1}{2}\right)^2.
\]
Let us first consider the multilevel Toeplitz matrices $T_\bfnn\left(g\right)$, with $\bfnn=(n,n)$. According to the discussion in Subsection \ref{ssec:multigrid} and Theorem \ref{thm:toep_approx_prop}, in order to obtain a two-grid method that fulfills the approximation property for the linear systems
\[
T_\bfnn\left(g\right)x_\bfnn=b_\bfnn
\]
we can choose the following prolongation operator
\begin{equation}\label{eq:prolong_operator}
	P_{\bfnn,\bfkk}=T_\bfnn\left((2+2\cos(\theta_1))(2+2\cos(\theta_2))\right)K_\bfnn.
\end{equation}
By Proposition 3.4 in \cite{ST}, if we prove that
\begin{equation}\label{eqn:T_le_DT}
	T_\bfnn\left(g\right)\le
	T_\bfnn\left(g\right)^{\frac{1}{2}}
	\diag_{\bfnn}\left(a\right)
	T_\bfnn\left(g\right)^{\frac{1}{2}},
\end{equation}
then $P_{\bfnn,\bfkk}$ is a suitable prolongation strategy also for the linear systems
\[
	T_\bfnn\left(g\right)^{\frac{1}{2}}
	\diag_{\bfnn}\left(a\right)
	T_\bfnn\left(g\right)^{\frac{1}{2}}x_\bfnn=b_\bfnn.
\]
By the Sylvester's law of inertia, condition \eqref{eqn:T_le_DT} is equivalent to
\[
I_\bfnn\le \diag_{\bfnn}\left(a\right)
\]
and the latter inequality is fulfilled if and only if the function $a(x,y)$ is greater than or equal to 1 on $[0,1]^2$, which is trivially verified. According to Proposition 3.5 in \cite{ST}, in order to choose a relaxation parameter $\omega$ such that the Richardson method satisfies the smoothing property we need to estimate the following quantity
\begin{align*}
\sup_{\bfnn} \rho\left(	T_\bfnn\left(g\right)^{\frac{1}{2}}
	\diag_{\bfnn}\left(a\right)
	T_\bfnn\left(g\right)^{\frac{1}{2}}\right)
	&=\sup_{\bfnn} \left\|	T_\bfnn\left(g\right)^{\frac{1}{2}}
	\diag_{\bfnn}\left(a\right)
	T_\bfnn\left(g\right)^{\frac{1}{2}}\right\|_2\\
	&\le \sup_{\bfnn} \left\|	T_\bfnn\left(g\right)^{\frac{1}{2}}\right\|_2
	\left\|\diag_{\bfnn}\left(a\right)\right\|_2
	\left\|T_\bfnn\left(g\right)^{\frac{1}{2}}\right\|_2\\
	&< \sqrt{8}\cdot\frac{5}{4}\cdot\sqrt{8}=10.
\end{align*}
Then, we can take $\omega$ in the interval $\left(0, \frac{2}{10}\right]$.
In order to study the spectral features of the considered matrix-sequence we employ the Generalized Locally Toeplitz (GLT) theory tools (\cite{glt-book-1}). Following the notation of \cite{twin-theory}, we use the GLT properties \textbf{GLT1}, \textbf{GLT3} and \textbf{GLT4} to show that the subsequent spectral distribution holds
\[
	\left\{T_\bfnn\left(g\right)^{\frac{1}{2}}
	\diag_{\bfnn}\left(a\right)
	T_\bfnn\left(g\right)^{\frac{1}{2}}\right\}_\bfnn \sim_{\lambda} a(x,y)g(\theta_1,\theta_2),
\]
where, in particular, we exploited the algebra property of GLT sequences, which include Toeplitz and diagonal-sampling matrix-sequences, provided that they fulfill specific hypothesis. By the very definition of spectral distribution, it is immediate to see that the latter is equivalent to
\[
	\left\{T_\bfnn\left(g\right)^{\frac{1}{2}}
	\diag_{\bfnn}\left(a\right)
	T_\bfnn\left(g\right)^{\frac{1}{2}}\right\}_\bfnn \sim_{\lambda} \mathfrak{f}(x,y,\theta_1,\theta_2)
\]
where $\mathfrak{f}(x,y,\theta_1,\theta_2)$ is defined as in (\ref{eq:distribution_FD_example}).

Consider now the multi-index $\bfnn=(n,n)$ such that $n=2^t$, $t\in\mathbb{N}$. From the definition of $V_{\bfnn}$ in equation (\ref{eq:def_Vn}), we recall that the dimension $d_\bfnn$ of the graph Laplacian $\Delta_{G_\bfnn}$ is such that $d_\bfnn< n^2$. Taking inspiration from all the previous considerations, we propose a two-grid method for the linear system
\[
	\Delta_{G_\bfnn}x_\bfnn=b_\bfnn
\]
having a prolongation operator of the form in equation (\ref{eq:prolong_operator}). More precisely, as a first step construct
\[
P_{\bfnn,\bfkk}=T_\bfnn\left((2+2\cos(\theta_1))(2+2\cos(\theta_2))\right)K_\bfnn
\]
with $\bfkk=(\frac{n}{2},\frac{n}{2})$. Then, eliminate from the matrix $P_{\bfnn,\bfkk}$ the rows $\mathbf{i}=(i_1,i_2)$ such that $(x_{i_1},y_{i_2})$ does not belong to $V_{\bfnn}$ and eliminate the columns $\mathbf{j}=(j_1,j_2)$ such that $(x_{j_1},y_{j_2})$ does not belong to $V_{\bfkk}$. With this procedure, the matrix-sizes are consistent and it is also possible to implement a V-cycle algorithm using the same prolongation/restriction strategy at all levels.

In Table \ref{table:FD_MGM} we numerically show the validity of our proposed methods, reporting the number of iterations needed for achieving the tolerance $\varepsilon = 10^{-6}$ when increasing the matrix size. In the first column, we show the results relative to the two-grid method with pre-smoother and post-smoother one iteration of the Richardson method with relaxation parameters $1/5$ and $2/15$ respectively. In the second and third columns, we use Gauss-Seidel for the two-grid and V-cycle algorithms respectively. In all cases, we see that the number of iterations needed for convergence remains almost constant, when increasing the size $d_\bfnn$.

\begin{table}[htb]
	\begin{center}
		\begin{tabular}{cccccccccc}
			\hline
	{$t$}&{$d_\bfnn$}& {Two-grid/Richardson} & {Two-grid/Gauss Seidel}& {V-cycle/Gauss Seidel}\\
			\hline
	     	3&	 60&15& 7& 7\\
			4& 	216&15& 9& 9\\
			5& 	848&17&10&10\\
			6& 3300&17& 9&10\\
			\hline
		\end{tabular}
	\end{center}
	\caption{{Comparison of the number of iterations, varying $t$, for the two-grid and V-cycle methods applied to the graph Laplacian $\Delta_{G_\bfnn}$.}}
	\label{table:FD_MGM}
\end{table}

\subsection{Elliptic problems on diamond Toeplitz graphs}
All the examples in the previous subsections were about graphs arising from PDEs approximations, and the PDEs were of the form of elliptic equations on bounded subsets of $\R^d$. For this numerical example, we move into a complete graph setting, where the graph can model many other different type of real-world interactions without being necessarily restricted to be the physical approximation of a geometric overlying domain.

Given a graph $G$ with node set $V$ such that $V=V_1\cup V_2$, $V_1\cap V_2 = \emptyset$, then the model problem we want to solve is 
\begin{equation}\label{eq:full-graph}
\begin{cases}
	\Delta_G u(v_i)= f(v_i) &\mbox{if } v_i \in V_1,\\
	u(v_i)= h(v_i) &\mbox{if } v_i \in  V_2.
\end{cases}	
\end{equation}
This kind of problem can be viewed again as an elliptic, nonhomogeneous Dirichlet problem (NHDP), but on graphs. Indeed, $\Delta_G$ is positive semi-definite, and if $G$ is a subgraph of a graph $\bar{G}$ with node set $\bar{V}$, then we can fix $V_1 = \mathring{V}$ and $V_2= \partial V$, where we recall that the (inner) boundary is defined as $\partial V:= \left\{ v_i \in V \, : \,  v_i \sim \bar{v}_j \mbox{ for some } \bar{v}_j\in \bar{V}\setminus V \right\}$ and $\mathring{V} = V\setminus \partial V$. With this choice, Problem \eqref{eq:full-graph} becomes
\begin{equation}\tag{NHDP}\label{eq:full-graph2}
	\begin{cases}
		\Delta_G u(v_i)= f(v_i) &\mbox{if } v_i \in \mathring{V},\\
		u(v_i)= h(v_i) &\mbox{if } v_i \in  \partial V.
	\end{cases}	
\end{equation}
For simplicity, we will assume now that both $G$ and $\bar{G}$ have zero potential terms, that is, $\kappa= \bar{\kappa}\equiv 0$. If we explicit now the action of the graph Laplacian, we can see that, for every $v_i \in \mathring{V}$, 
\begin{align*}
\Delta_G u(v_i) &= \sum_{\substack{v_j\sim v_i\\v_j \in V}} w(v_i,v_j)\left(u(v_i) - u(v_j)\right)\\
 &= \sum_{\substack{v_j\sim v_i\\v_j \in \mathring{V}}} w(v_i,v_j)\left(u(v_i) - u(v_j)\right) + \sum_{\substack{v_j\sim v_i\\v_j \in \partial V}} w(v_i,v_j)\left(u(v_i) - h(v_j)\right)\\
&= \sum_{\substack{v_j\sim v_i\\v_j \in \mathring{V}}} w(v_i,v_j)\left(u(v_i) - u(v_j)\right) + \mathring{\kappa}(v_i)u(v_i) - g(v_i), 
\end{align*}
where
$$
\mathring{\kappa}(v_i)=\sum_{\substack{v_j\sim v_i\\v_j \in \partial V}}w(v_i,v_j), \qquad g(v_i)= \sum_{\substack{v_j\sim v_i\\v_j \in \partial V}} w(v_i,v_j)h(v_j).
$$
Observing that
$$
\sum_{\substack{v_j\sim v_i\\v_j \in \mathring{V}}} w(v_i,v_j)\left(u(v_i) - u(v_j)\right) + \mathring{\kappa}(v_i)u(v_i)
$$
defines the action of the graph Laplacian associated with the graph $\mathring{G}:=\left(\mathring{V},  \mathring{E}, \mathring{w},\mathring{\kappa}\right)\subset G$, where 
$$
\mathring{E}= \left\{ (v_i,v_j) \in E \, | \, v_i, v_j \in \mathring{V} \right\}, \qquad \mathring{w} = w_{|\mathring{E}}, 
$$
then Problem \eqref{eq:full-graph2} is equivalent to solve
\begin{equation}\tag{NHDP'}\label{eq:full-graph3}
\Delta_{\mathring{G}}u(v_i) = g(v_i)+f(v_i), \quad	v_i \in \mathring{V}.
\end{equation}
Let us notice that the Dirichlet BCs in Problem \eqref{eq:full-graph2} have been absorbed by the forcing term $g$ and the potential term $\mathring{\kappa}$, and that $g\equiv 0$ if $h \equiv 0$ (zero Dirichlet BCs). Since 
$$
\mathring{\kappa}(v_i) = \sum_{\substack{v_j\sim v_i\\v_j \in \partial V}}w(v_i,v_j) = \sum_{\substack{v_j\sim v_i\\v_j \in V\setminus \mathring{V}}}w(v_i,v_j),
$$
again, we recover the Dirichlet potential as in Subsections \ref{ssec:multigrid_triangle} and \ref{ssec:multigrid_disk}. $\mathring{\kappa}$ accounts for the edge deficiency of the node $v_i$ in $\mathring{V}$, seen as a node in $V$. It emerges from the computation of the action of the graph Laplacian associated to $G$, when imposing $u\equiv 0$ on $\partial V=V\setminus \mathring{V}$. 

For the numerical experiments, owing to the equivalence of \eqref{eq:full-graph2} and \eqref{eq:full-graph3}, we work directly with $\mathring{G}$. Let us fix the mold graph $M=\left([4], E, w\right)$ such that
$$
w(i,j)= \begin{cases}
	j-i & \mbox{if } $i=1$ \mbox{ and } j=2,3,4,\\
	0 &\mbox{otherwise}. 
\end{cases}
$$ 
The diamonds will be denoted by $M(k)= (V(k), E(k), w_k)\simeq M$. We set 
$$
L= \begin{bmatrix}
	10 & 0 & 0 &0 \\
	0 & 0 & 0 & 0\\
	0 & 0 & 0& 0\\
	0 & 1 & 0 & 0
\end{bmatrix}
$$
and $G = T_{n,4}^M\left\langle (1, L) \right\rangle$ the corresponding $1$-level diamond Toeplitz graph, see Figure \ref{fig:diamond_graph}. Finally, we set  
\begin{itemize}
\item $\partial V= V(1)\cup V(n)$, that is, $\partial V$ is given by the nodes that belong to the node sets of the first and the last diamonds. Remember that $V(k)=\{v_{(k,1)}, v_{(k,2)}, v_{(k,3)}, v_{(k,4)}\}$;
\item $h(v_{(1,1)})=0.5, h(v_{(1,2)})=0.25, h(v_{(1,3)})=h(v_{(1,4)})=0$;
\item $h(v_{(n,1)})=0.5, h(v_{(n,2)})=0.25, h(v_{(n,3)})= h(v_{(n,4)})=0$;
\item $f(v_{(k,i)})= \sin(ki)$.
\end{itemize}
The symbol associated to the graph Laplacian $\Delta_{\mathring{G}}$ is given by
$$
\bfff (\theta) = D  -\left[ W +(L+L^*)\cos(\theta) +(L - L^*) \i\sin(\theta)\right]\in \mathbb{R}^{4\times 4}, \qquad \theta\in [-\pi,\pi]
$$
where 
$$
W= \begin{bmatrix}
	0 & 1 & 2 & 3\\
	1 & 0 & 0 & 0\\
	2 & 0 & 0 & 0\\
	3 & 0 & 0 & 0
\end{bmatrix}, \qquad 
D= \begin{bmatrix}
26 & 0 & 0 & 0\\
0 & 2 & 0 & 0\\
0 & 0 & 2 & 0\\
0 & 0 & 0 & 4
\end{bmatrix},
$$
which gives in compact form
$$
\bfff (\theta) = \begin{bmatrix}
26-20\cos(\theta) & -1 & -2 & -3\\
-1 & 2 & 0 & -\cos(\theta)+\i\sin(\theta)\\
-2 & 0 & 2 & 0\\
-3 & -\cos(\theta)-\i\sin(\theta) & 0 & 4
\end{bmatrix}.
$$
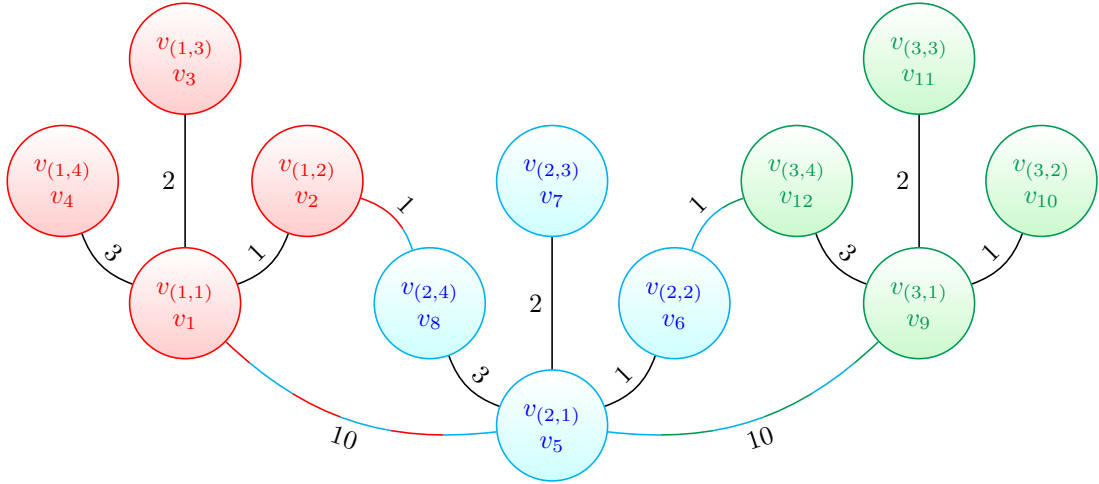
\begin{figure}
\centering
\begin {tikzpicture}[auto ,node distance =2.3 cm,on grid ,
semithick ,
state/.style ={ circle ,top color =white , bottom color = processblue!20 ,
	draw,processblue , text=blue , minimum width =1 cm},
state2/.style ={ circle ,top color =white , bottom color = ForestGreen!20 ,
	draw,ForestGreen , text=ForestGreen , minimum width =1 cm},
state3/.style ={ circle ,top color =white , bottom color = red!20 ,
	draw,red , text=red , minimum width =1 cm},
stateW/.style ={ circle ,top color =white , bottom color = white ,
	draw,white , text=white , minimum width =1 cm}]
\node[state] (A)  {\makecell[t]{$v_{(2,1)}$\\$v_5$}};
\node[state] (C) [above left=of A] {\makecell[t]{$v_{(2,4)}$\\$v_8$}};
\node[state] (B) [above right=of A] {\makecell[t]{$v_{(2,2)}$\\$v_6$}};
\node[state] (D) [above right=of C] {\makecell[t]{$v_{(2,3)}$\\$v_{7}$}};

\node[state2] (E) [above right=of B] {\makecell[t]{$v_{(3,4)}$\\$v_{12}$}};
\node[state2] (F) [above right=of E] {\makecell[t]{$v_{(3,3)}$\\$v_{11}$}};
\node[state2] (G) [below right=of F] {\makecell[t]{$v_{(3,2)}$\\$v_{10}$}};
\node[state2] (H) [below left=of G] {\makecell[t]{$v_{(3,1)}$\\$v_{9}$}};

\node[state3] (I) [above left=of C] {\makecell[t]{$v_{(1,2)}$\\$v_{2}$}};
\node[state3] (N) [above left=of I] {\makecell[t]{$v_{(1,3)}$\\$v_{3}$}};
\node[state3] (M) [below left=of N] {\makecell[t]{$v_{(1,4)}$\\$v_{4}$}};
\node[state3] (L) [below right=of M] {\makecell[t]{$v_{(1,1)}$\\$v_{1}$}};

\path (A) edge [black, bend left=25] node[sloped,midway] {$3$} (C);
\path (A) edge [black, bend right=25] node[sloped,midway] {$1$} (B);
\path (A) edge [black] node[midway] {$2$} (D);

\path (H) edge [black, bend left=25] node[sloped,midway] {$3$} (E);
\path (H) edge [black] node[midway] {$2$} (F);
\path (H) edge [black, bend right=25] node[sloped,midway] {$1$} (G);

\path (L) edge [black, bend right=25] node[sloped,midway] {$1$} (I);
\path (L) edge [black, bend left=25] node[sloped,midway] {$3$} (M);
\path (L) edge [black] node[midway] {$2$} (N);

\draw[-,bicolor={processblue}{ForestGreen},bend right=25]
(A) to node[sloped,midway, below] {$\textcolor{black}{10}$} (H);
\draw[-,bicolor={processblue}{red},bend left=25]
(A) to node[sloped,midway, below] {$\textcolor{black}{10}$} (L);

\draw[-,bicolor={processblue}{ForestGreen},bend left=27]
(B) to node[sloped, above] {$\textcolor{black}{1}$}  (E);
\draw[-,bicolor={red}{processblue},bend left=27]
(I) to node[sloped] {$\textcolor{black}{1}$} (C);
\end{tikzpicture}
\caption{Representation of the $1$-level diamond Toeplitz graph $G = T_{n,4}^M\left\langle (1, L) \right\rangle$ for $n=3$.}\label{fig:diamond_graph}
\end{figure}
In order to study the zeros of the four eigenvalue functions of $\bfff (\theta)$, we numerically checked that $0\le\lambda_1(\bfff (\theta))<\lambda_2(\bfff (\theta))<\lambda_3(\bfff (\theta))<\lambda_4(\bfff (\theta))$ for all $\theta\in [-\pi,\pi]$ and we then computed the determinant, which is equal to $\det(\bfff (\theta))=292 - 292\cos(\theta)$. Hence, we deduce that both the determinant and $\lambda_1(\bfff (\theta))$ have a zero of order 2 in 0. As a consequence, we define a grid transfer operator of the form (\ref{eq:def_projector_pnk}), associated to the following matrix-valued trigonometric polynomial
$$
\mathbf{p} (\theta) = (4+6\cos(\theta)+4\cos(2\theta)+2\cos(3\theta))\begin{bmatrix} 2 & 1 & 1 & 1 \\ 1 & 2 & 1 & 1 \\ 1 & 1 & 2 & 1 \\ 1 & 1 & 1 & 2 \\ \end{bmatrix}.
$$
Indeed, the trigonometric polynomial $2+6\cos(\theta)+4\cos(2\theta))+2\cos(3\theta)$ has zeros of order 2 in $\pi/2$ and $\pi$, which are the mirror points of 0 if we consider a coarsening factor of 4, as in Theorem \ref{thm:toep_approx_prop}. Moreover, the matrix
$$
\begin{bmatrix} 2 & 1 & 1 & 1 \\ 1 & 2 & 1 & 1 \\ 1 & 1 & 2 & 1 \\ 1 & 1 & 1 & 2 \\ \end{bmatrix}
$$
guarantees a constant convergence rate of the V-cycle method, according to the analysis in \cite{multi-block}.

In Table \ref{table:diamond} we numerically show the validity of our proposed methods, reporting the number of iterations needed for achieving the tolerance $\varepsilon = 10^{-6}$ when increasing the matrix size. In this case, we use the relative error stopping criterion, where the exact solution is computed using the LAMG library (v.2.2.1), whose run time and storage were empirically demonstrated to scale linearly with the number of edges \cite{LAMG}.

We use one iteration of Gauss-Seidel both as pre-smoother and post-smoother. In the first and second columns, we show the results relative to the two-grid and V-cycle methods for $\mathfrak{g}=2$, while in the third and fourth columns we consider a coarsening factor $\mathfrak{g}=4$. 

In all cases, we see that the number of iterations needed for convergence remains almost constant, when increasing the size $d_\bfnn$.

\begin{table}[htb]
	\begin{center}
		\begin{tabular}{cccccccccc}
			\toprule
	{$t$}&{$d_\bfnn$}& {Two-grid ($\mathfrak{g}=2$)} & {V-cycle ($\mathfrak{g}=2$)} & {Two-grid ($\mathfrak{g}=4$)}& {V-cycle ($\mathfrak{g}=4)$}\\
			\midrule
			4& 	1016&5&6& 16&16\\
			5& 	4088&5&6&20&22\\
			6& 16376&5&6&19&23\\
			7& 65528&5&6&19&24\\
			8&262136&5&6&20&25\\
			\bottomrule
		\end{tabular}
	\end{center}
	\caption{{Comparison of the number of iterations, varying $t$, for the two-grid and V-cycle methods applied to the graph Laplacian $\Delta_{G_\bfnn}$.}}
	\label{table:diamond}
\end{table}

\subsection{FEM and IgA approximations of elliptic PDEs}

In this subsection we perform some numerical tests concerning examples reported in \cite[Subsections 7.2, 7.3]{twin-theory}. For the sake of completeness we briefly recall the setting.\\

We consider the model problem
\begin{equation}\label{problem}
-\Delta u = g  \text{ in }\ \Omega,
\end{equation}
where $ \Omega=(0,1) $ and $ g \in \text{L}^2(\Omega) $. We first approximate \eqref{problem} by using quadratic FEs over the uniform mesh with stepsize $ \frac{1}{n+1} $, where we choose as FEs basis the quadratic $C^0$ B-spline basis over the knot sequence $ \{ \frac{1}{n+1}, \frac{1}{n+1}, \frac{2}{n+1}, \frac{2}{n+1},..., \frac{n}{n+1}, \frac{n}{n+1} \} $. Fixed $ \bfnn=(n,n) $ and proceeding as in \cite{FEM-paper}, it is possible to show that, in this case, the approximation of the operator $ -n^{-1}\Delta $ is given by the graph Laplacian $ \Delta_{G_{\bfnn}} $ of a $1$-level diamond Toeplitz graph $ T_{\bfnn,2}^{G} \langle \left(1,L_1 \right) \rangle $, where
$$
W= \left( \begin{array}{cc}
0 & 2 \\
2 & 0
\end{array} \right)
$$
is the adjacency matrix of the mold graph $ G $ and
$$
L_1=\left( \begin{array}{cc}
0 & 0 \\
2 & 2
\end{array} \right).
$$

Note that (see \cite{twin-theory}) we have a nonzero potential term which depends on the choice of the boundary conditions.\\
It is then easy to see that the sequence $ \{ \Delta_{G_{\bfnn}} \} $ satisfies
$$
\{ \Delta_{G_{\bfnn}} \} \sim_{\lambda} \boldsymbol{\mathfrak{f}}(\theta),
$$
where $ \boldsymbol{\mathfrak{f}}: [-\pi,\pi] \to \mathbb{C}^{2 \times 2} $ is given by
$$
 \boldsymbol{\mathfrak{f}}(\theta)= \frac{1}{3} \left\{ \left( \begin{array}{cc}
4 & -2 \\
-2 & 8
\end{array} \right) + \left( \begin{array}{cc}
0 & -2 \\
-2 & -4
\end{array} \right) \cos (\theta) + \left( \begin{array}{cc}
0 & -2 \\
2 & 0
\end{array} \right) \i \sin (\theta)  \right\}.
$$

It is now well known from the relevant literature (see \cite{sp-gr,DBFS93,solvers15,solvers-bis15,our-sinum}) that an optimal preconditioner for the linear system
\begin{equation}\label{linear_system}
\Delta_{G_{\bfnn}} u_{\bfnn} = g_{\bfnn}
\end{equation}
is given by the matrix $ S_n = C_n(\boldsymbol{\mathfrak{f}})+ \frac{1}{2n}ee^T $, $ e $ being the column vector of dimension $ 2n $ of all ones. This is numerically confirmed in Table \ref{table:block_PCG_table}.

\begin{table}[H]
	\begin{center}
		\begin{tabular}{cccccccccc}
		\toprule
		{$t$}&{$d_\bfnn$($n=2^t$)}& {CG} & {PCG}\\
		\midrule
		6& 	64&95&4\\
		7& 128&$>100$& 4\\
		8&	 256&$>100$&4\\
		9& 	512&$>100$&5\\
		10& 1024&$>100$&5\\
		11& 2048&$>100$&5\\
		\bottomrule
	\end{tabular}
	\end{center}
	\caption{{Comparison of the number of iterations, varying $t$, by the CG method and the PCG method with  preconditioner $ S_n $ with tolerance $10^{-6}$.}}
	\label{table:block_PCG_table}
\end{table}

If we now consider the discretization of \eqref{problem} arising from an IgA approach as in \cite[Subsection 7.3]{twin-theory} (see also \cite{our-sinum}), the discretizing operator $ \Delta_{G_n} $ is not the graph Laplacian of one of the structures described in Section \ref{sec:intro}; however, the difference is due to the presence of local perturbations near the boundary nodes which do not affect the symbol function. In particular, it happens that the sequence $ \{ \Delta_{G_n} \} $ has the same symbol function as the graph Laplacian of the $1$-level Toeplitz graph
$$
H_n=\left(T_n\left\langle \left(1,\frac{30}{240}\right),\left(1,\frac{48}{240}\right),\left(1,\frac{2}{240}\right)\right\rangle, \kappa\right), \quad \kappa\equiv 0,
$$
that is
$$
\{\Delta_{G_n}\} \sim_{\lambda} \mathfrak{f}(\theta),
$$
with $\mathfrak{f}(\theta)= \frac{160}{240} - \frac{60}{240}\cos(\theta) - \frac{96}{240}\cos(2\theta) - \frac{4}{240}\cos(3\theta)$.
By direct computation, it is easy to see that $ \mathfrak{f} $ has a unique zero of order $2$ at $0$. It follows that, if we choose $ p(\theta)= 2-2\cos(\theta)$, we deduce that $ \frac{\mathfrak{f}(\theta)}{p(\theta)}$ is monotonic decreasing in $[0,\pi]$, globally even, and
$$
\lim_{\theta \to 0} \frac{\mathfrak{f}(\theta)}{p(\theta)} = 1.
$$
The previous formula suggests as preconditioner for the PCG the matrix $ P_n=T_n(2-2\cos(\theta)) $, for which, following the theory developed in \cite{DBFS93}, we have
\begin{eqnarray*}
	\lim_{n \to \infty} \lambda_{\min}(P_n^{-1} A_n) & = & \min_{\theta\in [0,\pi]}  \frac{\mathfrak{f}(\theta)}{p(\theta)}
	=  \frac{\mathfrak{f}(\pi)}{p(\pi)}=\frac{2}{15} \\
	\lim_{n \to \infty} \lambda_{\max}(P_n^{-1} A_n) & = & \sup_{\theta\in [0,\pi]}  \frac{\mathfrak{f}(\theta)}{p(\theta)}
	= \lim_{\theta \to 0} \frac{\mathfrak{f}(\theta)}{p(\theta)} = 1.
\end{eqnarray*}

As we show in Table \ref{table:PCG_table1}, the PCG with preconditioner $ P_n = T_n(2-2\cos(\theta)) $ for the solution of the linear system \eqref{linear_system} is in fact optimal.

\begin{table}[H]
	\begin{center}
		\begin{tabular}{cccccccccc}
			\toprule
			{$t$}&{$d_\bfnn$($n=2^t$)}& {CG} & {PCG}\\
			\midrule
			7& 128&72& 20\\
			8&256 &$>100$&20\\
			9& 	512&$>100$& 21\\
		   10& 	1024&$>100$&21\\
		   11& 2048&$>100$&22\\
		   12& 4096&$>100$&22\\
			\bottomrule
		\end{tabular}
	\end{center}
	\caption{{Comparison of the number of iterations, varying $t$, for the IgA example by the CG method and the PCG method with  preconditioner $ P_n $ with tolerance $10^{-6}$ .}}
	\label{table:PCG_table1}
\end{table}

In what follows, we show an implementation of the MGM with a projector matrix whose choice depends on the analytic properties of the symbol function $ f $. We consider now the same examples as above in this subsection. In particular we emphasize that our analysis allows us to define a projector for which the MGM turns out to be an optimal method  when solving  system \eqref{linear_system}. Further, since the discrete operator $ \Delta_{G_n} $ is HPD, we use the Richardson method both as pre-smoother and post-smoother.
\\
We start with the IgA example and assume that $ n = 2^t-1 $, $ t \geq 1 $. Following the ideas of Subsection \ref{ssec:multigrid}, we can build our projector as $ P_n = T_n(p(\theta)) K_n $, with $ p(\theta) = 2+2 \cos(\theta) $ and $K_n$ the cutting matrix defined as in \eqref{eq:def_cutting_matrix}.
Recall that $ f $ has a unique zero in $ [0,\pi] $ attained at $ 0 $. It follows that the trigonometric polynomial $ p $ defined above is such that
$$
\lim_{x \to 0} \frac{p^2(\pi-x)}{f(x)} < \infty
$$
and
$$
p^2(x) + p^2(\pi-x) > 0 \quad \forall x \in [0,\pi],
$$
where the two conditions above insure the optimality in light of Theorem \ref{thm:toep_approx_prop}. As noted above, we know that $ \Delta_{G_n} $ has not the structure of the graph Laplacian of a Toeplitz graph. We denote by $ R_n$ the difference $ \Delta_{G_n}-\Delta{H_n} $, where $ \Delta{H_n} $ is the graph Laplacian of the Toeplitz graph $H_n$ in the previous subsection, which is a low rank correction. In order to ensure convergence, in numerical examples we use $0<\omega<\frac{2}{M_f+\rho(R_n)}$ as the scalar parameter in Richardson method one iteration of the Richardson method with relaxation parameters $0.7149$ and $1.4299$ respectively, where $M_f\approx 0.5333$ is a maximum value of $f(\theta)$ and $\rho(R_n)\approx 0.8654$ is the spectral radius of $R_n$, thus satisfying condition in Proposition 3.5 in \cite{ST}.  In Table \ref{table1:FD_MGM} we numerically show the validity of our proposed methods.

\begin{table}[H]
	\begin{center}
		\begin{tabular}{cccccccccc}
			\toprule
			{$t$}&{$d_\bfnn$($n=2^t-1$)}& {Two-grid Method}\\
			\midrule
			7& 127&8\\
			8&255 &8\\
			9& 	511&8\\
			10&1023&8\\
			11& 2047&8\\
			12& 4095&8\\
			\bottomrule
		\end{tabular}
	\end{center}
	\caption{Number of iterations, varying $t$, for the two-grid method applied to the IgA example with tolerance $10^{-6}$.}
	\label{table1:FD_MGM}
\end{table}

An analogous study can be carried for the discretization via the FEM approach.\\

In this case we assume $ n=2^{t} $, with $ t $ integer. By direct computation, it is possible to show that the same projector $ P_n $ can be used also in this situation; we only need to pay attention to the dimension of the projector (which is going to be $ 2n \times n $ due to the block structure of the matrix involved) and, as above, we use both as pre-smoother and post-smoother one iteration of the Richardson method with relaxation parameters $0.25$ and $0.50$ respectively. In Table \ref{table2:FD_MGM}, we show the optimality of the considered method.\\

\begin{table}[H]
	\begin{center}
		\begin{tabular}{cccccccccc}
\toprule		
			{$t$}&{$d_\bfnn$($n=2^t$)}& {Two-grid Method}\\
			\midrule
			7& 128&9\\
			8&256 &9\\
			9& 	512&9\\
			10&1024&9\\
			11& 2048&9\\
			12& 4096&9\\
\bottomrule
		\end{tabular}
	\end{center}
	\caption{Number of iterations, varying $t$, for the two-grid method applied to the FEs example with tolerance $10^{-6}$.}
	\label{table2:FD_MGM}
\end{table}

\section{Conclusions}\label{sec:conclusion}

In the present work we have treated sequences of graphs having a grid geometry, with a uniform local structure in a bounded domain $\Omega\subset {\mathbb R}^d$, $d\ge 1$. When $\Omega=[0,1]$, such graphs include the standard Toeplitz graphs and, for $\Omega=[0,1]^d$,  the considered class includes $d$-level Toeplitz graphs. In the general case, the underlying
sequence of adjacency matrices has a canonical eigenvalue distribution, in the Weyl sense, and it has been shown in the theoretical part of this work that we can associate to it a symbol $\boldsymbol{\mathfrak{f}}$, also in the case of variable coefficients in connection with the notion of Generalized Locally Toeplitz sequences.

Here we have given practical evidence that the knowledge of the symbol and of its basic analytical features provided enough information on the eigenvalue structure in terms of  localization, spectral gap, clustering, and global distribution, in order to design efficient numerical methods for the corresponding large linear systems and for approximating in a fast way the eigenvalue of continuous differential operators.  Tests and applications have been taken from the approximation of differential operators via (local) numerical schemes such as Finite Differences (FDs), Finite Elements (FEs), and Isogeometric Analysis (IgA), where the differential operator domains are of non-cartesian nature. 

Nevertheless, more applications can be taken into account, since the results presented here can be applied as well to study the spectral properties of adjacency matrices and Laplacian operators of general large graphs and networks, whenever the involved matrices enjoy a uniform local structure. The extension of the study concerning this last issue will be a specific direction to be investigated in future works.


\begin{thebibliography}{99}




\bibitem{twin-theory} {\sc A. Adriani, D. Bianchi, S. Serra-Capizzano},
{\itshape Asymptotic spectra of large (grid) graphs with a uniform local structure (part I): theory}. Milan J. Math. 88: 409--454 (2020).

\bibitem{ADS} {\sc A. Aric\`o, M. Donatelli, S. Serra-Capizzano}, {\itshape V-cycle optimal convergence for certain (multilevel) structured linear systems.} SIAM J. Matrix Anal. Appl. {26}:  186--214 (2004).

    

\bibitem{Barbarino} {\sc G. Barbarino}, {\itshape A systematic approach to reduced GLT.} BIT: 1--63 (2021).

\bibitem{Barbarino-Bianchi-Garoni} {\sc G. Barbarino, D. Bianchi, C. Garoni}, {\itshape Constructive approach to the monotone rearrangement of functions.} Expo. Math. (2021), in press.

\bibitem{glt-book-3} {\sc G. Barbarino, C. Garoni, S. Serra-Capizzano}, {\itshape Block Generalized Locally Toeplitz sequences: theory and applications in the unidimensional case.} Electron. Trans. Numer. Anal. 53: 28--112 (2020).

\bibitem{glt-book-4} {\sc G. Barbarino, C. Garoni, S. Serra-Capizzano}, {\itshape  Block Generalized Locally Toeplitz sequences: theory and applications in the multidimensional case.} Electron. Trans. Numer. Anal. 53: 113--216 (2020). 


\bibitem{Bianchi} {\sc D. Bianchi}, {\itshape Analysis of the spectral symbol function for discretization of a linear differential operator and analysis of the relative spectrum, with applications}. Calcolo 58: paper 38, 47 pp. (2021).


\bibitem{Bianchi-Donatelli} {\sc D. Bianchi, M. Donatelli}, {\itshape Graph approximation and generalized Tikhonov regularization for signal deblurring}. Proceedings in IEEE ICCSA (2021): The International Conference on Computational Science and its Applications. To appear.


\bibitem{BS18} {\sc D. Bianchi, S. Serra-Capizzano}, {\itshape Spectral analysis of finite-dimensional approximations of 1d waves in non-uniform grids}. Calcolo 55: paper 47, 28 pp. (2018).

\bibitem{BSW} {\sc D. Bianchi, A.G. Setti, R.K. Wojciechowski.} {\itshape The generalized porous medium equation on graphs: existence and uniqueness of solutions with $\ell^1$ data}. Preprint.

\bibitem{BS} {\sc A. B\"ottcher, B. Silbermann}, {\itshape Introduction to
Large Truncated Toeplitz Matrices}. Springer-Verlag, New York
(1999).


\bibitem{BD} {\sc B. Brosowski, F. Deutsch}, {\itshape An elementary proof of the Stone-Weierstrass Theorem.} Proc. Amer. Math. Soc. 81(1) (1981).




\bibitem{FE-book} {\sc P. Ciarlet}, {\itshape The Finite Element Method for Elliptic Problems}.  North Holland, Amsterdam (1978).

\bibitem{IgA-book} {\sc J.A. Cottrell, T.J.R. Hughes,  Y. Bazilevs}, {\itshape   Isogeometric analysis: toward integration of CAD and FEA}. John Wiley \& Sons (2009).

\bibitem{sp-gr} {\sc D. Cvetkovic, M. Doob, H. Sachs}, {\itshape Spectra of Graphs}. Academic Press, New York (1979).

\bibitem{Davies} {\sc E.B. Davies}, {\itshape Spectral theory and differential operators}. Cambridge University Press 42 (1996).

\bibitem{DBFS93} {\sc F. Di Benedetto, G. Fiorentino, S. Serra-Capizzano}, {\itshape CG preconditioning for Toeplitz matrices}. Comput. Math. Appl. 25(6): 35--45 (1993).

\bibitem{solvers15} {\sc M. Donatelli, C. Garoni, C. Manni, S. Serra-Capizzano, H. Speleers}, {\itshape Robust and optimal multi-iterative techniques for IgA Galerkin linear systems.} Comput. Methods Appl. Mech. 284: 230--264 (2015).

\bibitem{solvers-bis15} {\sc M. Donatelli, C. Garoni, C. Manni, S. Serra-Capizzano, H. Speleers}, {\itshape Two-grid optimality for Galerkin linear systems based on B-splines.} Comput. Vis. Sci. 17: 119--133 (2015).

\bibitem{our-sinum} {\sc M. Donatelli, C. Garoni, C. Manni, S. Serra-Capizzano, H. Speleers}, {\itshape Symbol-based multigrid methods for Galerkin B-spline isogeometric analysis.} SIAM J. Numer. Anal. 55(1): 31--62 (2017).

\bibitem{multi-block} {\sc M. Donatelli, P. Ferrari, I. Furci, S. Serra-Capizzano, D. Sesana,}  {\itshape Multigrid methods for block-Toeplitz linear systems: convergence analysis and applications.} {Numer. Linear Algebra Appl.} {28(4): e2356 (2021)}.

\bibitem{DMS} {\sc M. Donatelli, M. Mazza, S. Serra-Capizzano,}  {\itshape Spectral analysis and multigrid methods for finite volume approximations of space-fractional diffusion equations.} {SIAM J. Sci. Comput.} {40(6): A4007--A4039 (2018)}.

\bibitem{DSS} {\sc M. Donatelli, S. Serra-Capizzano, D. Sesana,}  {\itshape Multigrid methods for Toeplitz linear systems with different size reduction.} {BIT} {52(2): 305--327 (2012)}.

\bibitem{EFGMSS}{\sc S.-E. Ekström, I. Furci, C. Garoni, C. Manni, S. Serra-Capizzano, H. Speleers,} {\itshape Are the eigenvalues of the $B$-spline isogeometric analysis approximation of $-\Delta u=\lambda u$ known in almost closed form?} Numer. Linear Algebra Appl. 25(5) (2018): e2198.








\bibitem{Estrada} {\sc E. Estrada}, {\itshape Path Laplacian matrices: introduction and application to the analysis of consensus in networks}. Linear Algebra Appl., 436: 3373--3391 (2012).








\bibitem{Fiorentino-Serra} {\sc G. Fiorentino, S. Serra-Capizzano}, {\itshape Multigrid methods for Toeplitz matrices}. Calcolo {28}: 283--305 (1991).




\bibitem{GMS18} {\sc C. Garoni, M. Mazza, S. Serra-Capizzano}, {\itshape Block Generalized Locally Toeplitz sequences: from the theory to the applications.} Axioms 7(3): 49 (2018).



\bibitem{glt-book-1} {\sc C. Garoni, S. Serra-Capizzano}, {\itshape The theory of Generalized Locally Toeplitz sequences:
 theory and applications - Vol I}.  Springer  - Springer Monographs in Mathematics, New York (2017).

\bibitem{glt-book-2} {\sc C. Garoni, S. Serra-Capizzano}, {\itshape The theory of multilevel Generalized Locally Toeplitz sequences:  theory and applications - Vol II}.  Springer  - Springer Monographs in Mathematics, New York (2018).



\bibitem{FEM-paper} {\sc C. Garoni, S. Serra-Capizzano, D. Sesana},
{\itshape Spectral analysis and spectral symbol of $d$-variate $\mathbb{Q}_{\boldsymbol p}$ Lagrangian FEM stiffness matrices}.
 SIAM J. Matrix Anal. Appl. 36(3): 1100--1128 (2015).

\bibitem{glt-book-3-old} {\sc C. Garoni, S. Serra-Capizzano, D. Sesana}, {\itshape  The Theory of Block Generalized Locally Toeplitz Sequences}.  
Technical Report, N. 1, January 2018, Department of Information Technology, Uppsala University.



\bibitem{GS} {\sc U. Grenander, G. Szeg\"o}, {\itshape Toeplitz Forms and Their Applications}. 2nd ed., Chelsea, New York (1984).


\bibitem{Ghorban2012}{\sc S. Hossein Ghorban}, {\itshape Toeplitz graph decomposition}. Trans. Combinat. 1(4): 35--41 (2012).

\bibitem{LAMG}{\sc O.E. Livne, A. Brandt}, {\itshape Lean algebraic multigrid (LAMG): fast graph Laplacian linear solver}. SIAM J. Sci. Comput. 34(4): 499--522 (2012).

\bibitem{bookColl}{\sc J. Lund, K.L. Bowers}, {\itshape Sinc methods for quadrature and differential equations}. SIAM, (1992).






\bibitem{Keller2021} {\sc M. Keller, D. Lenz, R.K. Wojciechowski}, {\itshape Graphs and discrete Dirichlet spaces}. Grundlehren der mathematischen Wissenschaften. To appear.






\bibitem{Oost} {\sc C. W. Oosterlee, A. Sch{\"u}ller, U. Trottenberg}, {\itshape Multigrid}. Academic Press, London (2001).



\bibitem{RStub} {\sc J.W. Ruge, K. St\"{u}ben}, {\itshape Multigrid methods}. SIAM, Philadelphia, PA (1987).





\bibitem{saad} {\sc Y. Saad}, {\itshape Iterative methods for sparse linear systems.} $2^{\rm nd}$ Edition, SIAM, Philadelphia, PA, (2003).












\bibitem{Sun} {\sc S. Serra-Capizzano},
{\itshape Convergence analysis of two-grid methods for elliptic Toeplitz and PDEs matrix-sequences.} Numer. Math. {92(3)}: 433--465  (2002).

\bibitem{glt} {\sc S. Serra-Capizzano}, {\itshape Generalized Locally
Toeplitz sequences: spectral analysis and applications to
discretized partial differential equations}. Linear Algebra Appl.
366: 371--402 (2003).

\bibitem{glt-bis} {\sc S. Serra-Capizzano}, {\itshape The GLT class as a
generalized Fourier analysis and applications}, Linear Algebra Appl. 419: 180--233 (2006).



\bibitem{ST} {\sc S. Serra-Capizzano, C. Tablino-Possio}, {\itshape {T}wo-Grid Methods for {H}ermitian positive definite   linear systems connected with an order relation}, Calcolo. 51(2): 261--285 (2014).







\bibitem{FD} {\sc J.C. Strikwerda}, {\itshape Finite Difference Schemes and
Partial Differential Equations}. Chapman and Hall, International
Thompson Publ., New York (1989).


\bibitem{Tilliloc} {\sc P. Tilli}, {\itshape Locally Toeplitz matrices: spectral
theory and applications}. Linear Algebra Appl. {278}: 91--120 (1998).


\bibitem{Tillinota} {\sc P. Tilli}, {\itshape A note on the spectral distribution of Toeplitz matrices}. Linear Multilin. Algebra {45}: 147--159 (1998).


\bibitem{tyrtL1} {\sc E. Tyrtyshnikov, N. Zamarashkin}, {\itshape Spectra of
multilevel Toeplitz matrices: advanced theory via simple matrix
relationships}. Linear Algebra Appl. 270: 15--27 (1998).


\end{thebibliography}
\end{document}